\documentclass{amsart}

\usepackage{amsmath}
\usepackage{amsfonts}
\usepackage{amssymb}
\usepackage{amsthm}
\usepackage{comment}
\usepackage{epsfig}
\usepackage{psfrag}
\usepackage{mathrsfs}
\usepackage{amscd}
\usepackage[all]{xy}
\usepackage{rotating}
\usepackage{lscape}
\usepackage{amsbsy}
\usepackage{verbatim}
\usepackage{moreverb}
\usepackage{color}
\usepackage{bbm}
\usepackage{eucal}
\usepackage{stmaryrd}

\usepackage{tikz-cd}
\usetikzlibrary{patterns,shapes.geometric,arrows,decorations.markings}
\usepackage{tikz-3dplot}

\usepackage{caption}
\usepackage{subcaption}

\colorlet{lightgray}{black!15}

\tikzset{->-/.style={decoration={
  markings,
  mark=at position .5 with {\arrow{>}}},postaction={decorate}}}
\tikzset{midarrow/.style={decoration={
    markings,
    mark=at position {#1} with {\arrow{>}}},postaction={decorate}}}

\newtheorem{theorem}{Theorem}[section]
\newtheorem{prop}[theorem]{Proposition}
\newtheorem{lemma}[theorem]{Lemma}
\newtheorem{cor}[theorem]{Corollary}

\theoremstyle{definition}
\newtheorem{definition}[theorem]{Definition}

\newtheorem{observation}[theorem]{Observation}

\newtheorem{remark}[theorem]{Remark}

\newtheorem{notation}[theorem]{Notation}
\newtheorem{l.notation}[theorem]{Local Notation}

\newtheorem{convention}[theorem]{Convention}

\theoremstyle{remark}

\definecolor{orange}{rgb}{.95,0.5,0}
\definecolor{light-gray}{gray}{0.75}
\definecolor{brown}{cmyk}{0, 0.8, 1, 0.6}
\definecolor{plum}{rgb}{.5,0,1}

\DeclareMathOperator{\ev}{\mathsf{ev}}

\DeclareMathOperator{\bBar}{\sf Bar}
\DeclareMathOperator{\Alg}{\sf Alg}

\DeclareMathOperator{\Mod}{\sf Mod}

\DeclareMathOperator{\CAlg}{\sf CAlg}

\DeclareMathOperator{\Aut}{\sf Aut}
\DeclareMathOperator{\colim}{{\sf colim}}

\DeclareMathOperator{\limit}{{\sf lim}}

\DeclareMathOperator{\Hom}{\sf Hom}
\DeclareMathOperator{\End}{\sf End}

\DeclareMathOperator{\Fun}{{\sf Fun}}

\DeclareMathOperator{\Map}{{\sf Map}}

\DeclareMathOperator{\exit}{\sf Exit}
\DeclareMathOperator{\Exit}{\bcE{\sf xit}}

\DeclareMathOperator{\Cat}{{\sf Cat}}
\DeclareMathOperator{\fCat}{{\sf fCat}}

\DeclareMathOperator{\Quiv}{\sf Quiv}

\DeclareMathOperator{\Ar}{{\sf Ar}}

\DeclareMathOperator{\Diff}{{\sf Diff}}

\DeclareMathOperator{\op}{\mathsf{op}}

\DeclareMathOperator{\cBun}{{\sf c}\cB\mathsf{un}}

\DeclareMathOperator{\sk}{\mathsf{sk}}

\DeclareMathOperator{\Mfd}{{\cM}\mathsf{fd}}
\DeclareMathOperator{\cMfd}{{\sf c}{\cM}\mathsf{fd}}
\DeclareMathOperator{\Mfld}{{\cM}\mathsf{fld}}

\DeclareMathOperator{\conf}{\mathsf{Conf}}

\DeclareMathOperator{\Spaces}{\cS\mathsf{paces}}

\DeclareMathOperator{\Disk}{\cD{\mathsf{isk}}}

\DeclareMathOperator{\Adj}{\mathsf{Adj}}

\DeclareMathOperator{\fr}{\sf fr}
\DeclareMathOperator{\sfr}{\sf sfr}

\DeclareMathOperator{\Bord}{{\sf Bord}_1^{\fr}}

\def\ot{\otimes}

\DeclareMathOperator{\oo}{\infty}

\DeclareMathOperator{\tr}{\triangleright}
\DeclareMathOperator{\tl}{\triangleleft}

\newcommand{\lag}{\langle}
\newcommand{\rag}{\rangle}

\newcommand{\un}{\underline}

\newcommand{\ra}{\rightarrow}

\newcommand{\xra}{\xrightarrow}
\newcommand{\xla}{\xleftarrow}

\def\cA{\mathcal A}\def\cB{\mathcal B}\def\cC{\mathcal C}\def\cD{\mathcal D}
\def\cE{\mathcal E}
\def\cK{\mathcal K}
\def\cM{\mathcal M}
\def\cS{\mathcal S}
\def\cV{\mathcal V}\def\cW{\mathcal W}\def\cX{\mathcal X}
\def\cY{\mathcal Y}

\def\DD{\mathbb D}

\def\RR{\mathbb R}\def\SS{\mathbb S}\def\TT{\mathbb T}

\def\ZZ{\mathbb Z}

\def\sB{\mathsf B}\def\sC{\mathsf C}
\def\sH{\mathsf H}
\def\sL{\mathsf L}
\def\sO{\mathsf O}
\def\sR{\mathsf R}\def\sS{\mathsf S}

\def\sZ{\mathsf Z}

\def\bdelta{\mathbf\Delta}
\def\bDelta{\mathbf\Delta}

\def\fB{\frak B}\def\fC{\frak C}

\def\bcD{\boldsymbol{\mathcal D}}\def\bcE{\boldsymbol{\mathcal E}}

\def\bcM{\boldsymbol{\mathcal M}}

\def\bcM{\boldsymbol{\mathcal M}}

\DeclareMathOperator{\Obj}{\mathsf{Obj}}

\DeclareMathOperator{\PShv}{\mathsf{PShv}}
\DeclareMathOperator{\uno}{\mathbbm{1}}

\DeclareMathOperator{\id}{\sf id}

\DeclareMathOperator{\Mor}{\sf Mor}
\DeclareMathOperator{\Dual}{\sf Dual}
\def\bfX{\boldsymbol{\mathfrak X}}

\DeclareMathOperator{\Sets}{\sf Sets}

\DeclareMathOperator{\Perf}{\sf Perf}
\DeclareMathOperator{\trace}{\sf trace}
\DeclareMathOperator{\unit}{\sf unit}

\DeclareMathOperator{\lax}{\sf lax}

\DeclareMathOperator{\para}{\bDelta_{\circlearrowleft}}
\newcommand{\bit}[1]{\textbf{\textit{#1}}}

\DeclareMathOperator{\racts}{\curvearrowright}
\DeclareMathOperator{\lacts}{\curvearrowleft}

\begin{document}

\title{Traces for factorization homology in dimension~1}

\author{David Ayala \& John Francis}

\address{Department of Mathematics\\Montana State University\\Bozeman, MT 59717}
\email{david.ayala@montana.edu}
\address{Department of Mathematics\\Northwestern University\\Evanston, IL 60208}
\email{jnkf@northwestern.edu}
\thanks{DA was supported by the National Science Foundation under awards 1812055 and 1945639. JF was supported by the National Science Foundation under award 1812057. This material is based upon work supported by the National Science Foundation under Grant No. DMS-1440140, while the authors were in residence at the Mathematical Sciences Research Institute in Berkeley, California, during the Spring 2020 semester.}

\begin{abstract}
We construct a circle-invariant trace from the factorization homology over the circle
\[
{\sf trace}
\colon
\int^\alpha_{\SS^1} \un\End(V) \longrightarrow \uno
\]
associated to a dualizable object $V\in \bfX$ in a symmetric monoidal $(\oo,1)$-category.
This proves a conjecture of To\"en--Vezzosi~\cite{toen.vezzosi} on existence of circle-invariant traces. 
Underlying our construction is a calculation of the factorization homology over the circle of the walking adjunction in terms of the paracyclic category of Getzler--Jones~\cite{getzler.jones}:
\[
\int_{\SS^1} \Adj 
~\simeq~
\para^{\tl\!\tr}
~.
\]
This calculation exhibits a form of Poincar\'e duality for 1-dimensional factorization homology.  
\end{abstract}



\maketitle

\tableofcontents

\section*{Introduction}

In this work, we generalize foundational algebraic results concerning traces and Hochschild homology established for chain complexes or spectra to symmetric monoidal $(\oo,1)$-categories. 
Our tool is the theory of factorization homology for $(\infty,1)$-categories, through which we apply the geometry of 1-dimensional manifolds to obtain these results in algebra.  
This theory of factorization homology for $(\infty,1)$-categories is developed in each of the logically independent works~\cite{fact1} and~\cite{circle}, and is related to factorization homology of associative algebras as developed in the work~\cite{oldfact}.
(See the following section, on implementation of factorization homology, for a discussion of these versions and how we use each in this paper.)

\smallskip

\subsubsection*{\bf Classical situation}
We briefly recall some classical results.
Let $\Bbbk$ be a commutative ring.\footnote{
In this discussion, everything is derived: we work in the symmetric monoidal $(\oo,1)$-category 
$(\Mod_{\Bbbk}, \underset{\Bbbk}\ot )$, which can be presented as an $\oo$-categorical localization of a category of chain complexes over $\Bbbk$ and derived tensor product over $\Bbbk$, localized on quasi-isomorphisms.
So, for instance, by the \emph{homology} $\sH(X;\Bbbk)$ of a space $X$, we mean the object in $\Mod_{\Bbbk}$ presented by $\Bbbk$-valued chains on $X$; by \emph{Hochschild homology} ${\sf HH}$, we mean Hochschild chains.
}
For each perfect $\Bbbk$-module $V$, \bit{trace} of $\Bbbk$-linear endomorphisms of $V$ defines a $\Bbbk$-linear map
\[
\trace
\colon
\un{\End}_\Bbbk(V)
\xla{~\simeq~}
V \underset{\Bbbk}\ot V^\vee
\xra{~\ev~}
\Bbbk
~,
\]
where $V^\vee$ is the $\Bbbk$-linear dual of $V$.
This trace map respects composition of $\Bbbk$-linear endomorphisms of $V$ in that there is a canonical identification $\trace(\varphi \circ \psi) \simeq \trace(\psi \circ \varphi)$ in $\Bbbk$.
More generally, $\trace$ is \bit{cyclically invariant}: there is a canonical identification in $\Bbbk$:
\[
\trace(\varphi_1\circ \cdots \circ \varphi_r) 
~\simeq~
\trace(\varphi_2\circ \cdots \circ \varphi_r\circ  \varphi_1)
~.
\]
The defining property of \bit{Hochschild homology} is just so that this cyclic invariance of $\trace$ organizes as the first of the following factorizations among $\Bbbk$-modules:
\begin{equation}\label{e21}
\xymatrix{
\un{\End}_\Bbbk(V)
\ar@(d,-)[drrrr]_-{\trace}
\ar[rr]
&&
{\sf HH}\bigl( \un{\End}_\Bbbk(V) \bigr)
\ar[rr]^-{{\sf HH}({\rm inclusion})}
\ar@{-->}[drr]^-{\trace}
&&
{\sf HH}\bigl( \Perf_\Bbbk \bigr)
\ar@{-->}[d]^-{\trace}
&&
&
{\sf HH}(\Bbbk)
\ar[lll]_-{\simeq}^-{\rm (Morita~invariance)}
\\
&&
&&
\Bbbk
&&
&
.
}
\end{equation}
Morita invariance of Hochschild homology grants the further 
factorization through the Hochschild homology of the (derived) $\Bbbk$-linear category of perfect $\Bbbk$-modules, as well as its identification with the Hochschild homology of $\Bbbk$.  
An advantage of this extension of $\trace$ to ${\sf HH}(\Perf_\Bbbk)$ is that there is a canonical $\Bbbk$-linear map from the $\Bbbk$-homology of the moduli space of perfect $\Bbbk$-modules:
\begin{equation}\label{e24}
{\sf H}\bigl( \Obj( \Perf_\Bbbk) ; \Bbbk \bigr)
\longrightarrow
{\sf HH}(\Perf_\Bbbk)
\xla{~\simeq~}
{\sf HH}(\Bbbk)
~.
\end{equation}
This map offers a composite $\Bbbk$-linear map:
\begin{equation}\label{e27}
{\sf H}\bigl( \Obj( \Perf_\Bbbk) ; \Bbbk \bigr)
\underset{(\ref{e24})}\longrightarrow
{\sf HH}(\Perf_\Bbbk)
\underset{(\ref{e21})}{\xra{~\trace~}}
\Bbbk
~
\simeq
~
\un{\End}_{\Bbbk}(\Bbbk)
~,
\end{equation}
which, by construction, evaluates as
\[
V
\longmapsto 
\trace( \id_V )
~\simeq~
\bigl( \Bbbk \xra{ \unit} \un{\End}_\Bbbk(V) \xla{\simeq} V \underset{\Bbbk}\ot V^\vee \xra{\ev} \Bbbk \bigr)
~,
\]
the trace of the identity, also known as the \bit{dimension} of $V$.

\smallskip

\subsubsection*{\bf Classical $\TT$-invariance}
In these classical terms, we now recall invariance of $\trace$, in~(\ref{e21}), with respect to the circle group $\TT$.
Namely, there is a canonical \bit{$\TT$-action} on Hochschild homology.
As such, the $\Bbbk$-linear map~(\ref{e24}) canonically factors through the (derived) \bit{$\TT$-invariants}:
\[
\xymatrix{
&&
{\sf HH}(\Perf_\Bbbk)^{\TT}
\ar[d]
&&
{\sf HC}^-(\Bbbk) 
\ar[ll]_-{\simeq}
\ar[d]
\\
{\sf H}\bigl( \Obj( \Perf_\Bbbk) ; \Bbbk \bigr)
\ar[rr]^-{(\ref{e24})}
\ar@{-->}[urr]
&&
{\sf HH}(\Perf_\Bbbk)
&&
{\sf HH}(\Bbbk)
\ar[ll]_-{\simeq}
~,
}
\]
where the (derived) $\TT$-invariants of ${\sf HH}(\Bbbk)$ is the \bit{negative cyclic homology} of $\Bbbk$.
Meanwhile, the construction of the $\Bbbk$-linear trace map~(\ref{e21}) is canonically $\TT$-invariant with respect to the trivial $\TT$-action on $\Bbbk\simeq \un{\End}_{\Bbbk}(\Bbbk)$.
In particular, there is a $\Bbbk$-linear map between (derived) $\TT$-invariants:
\begin{equation}\label{e26}
{\sf HC}^-(\Bbbk)
~\simeq~
{\sf HH}(\Perf_\Bbbk)^{\TT}
\xra{~\trace^{\TT}~}
\Bbbk^{\TT}
~\simeq~
\Bbbk^{\pitchfork \sB\TT}
\simeq
\Bbbk[c]
~,
\end{equation}
where the codomain is the cohomology of $\sB\TT$ with coefficients in $\Bbbk$, which is a polynomial ring over $\Bbbk$ on a generator $c$ in degree $-2$.
There results a lift of~(\ref{e27}):
{\small
\begin{equation}\label{e28}
\xymatrix{
&
{\sf HH}(\Perf_\Bbbk)^{\TT}
\ar[r]^-{\trace^{\TT}}
\ar[d]
&
\un{\End}_{\Perf_\Bbbk}(\Bbbk)^{\pitchfork \sB \TT}
\ar[d]
\\
{\sf H}\bigl( \Obj( \Perf_\Bbbk) ; \Bbbk \bigr)
\ar[r]_-{(\ref{e24})}
\ar@{-->}[ur]
&
{\sf HH}(\Perf_\Bbbk)
\ar[r]_-{(\ref{e21})}^-{\trace}
&
\un{\End}_{\Perf_\Bbbk}(\Bbbk)
~,
&
V
\mapsto 
\Bigl(
\Bbbk
\xra{\unit}
V^\vee \underset{\Bbbk}\ot V
\simeq
V \underset{\Bbbk}\ot V^{\vee}
\xra{\ev}
\Bbbk
\Bigr)
~.
}
\end{equation}
}

\smallskip

\subsubsection*{\bf Generalization.}
As {\bf Theorem~\ref{r9}} of this article, we prove the following generalization. 
Let $\bfX$ be a rigid symmetric monoidal $(\oo,1)$-category, i.e., a symmetric monoidal $(\oo,1)$-category in which every object is dualizable. 
We construct a $\TT$-invariant trace map to its endomorphisms of its symmetric monoidal unit $\uno \in \bfX$ from its \bit{factorization homology} over the circle (i.e., its Hochschild homology):
\begin{equation}\label{e29}
{\sf HH}(\bfX)
~:=~
\displaystyle \int_{\SS^1} \bfX
\xra{~\trace~}
\End_{\bfX}(\uno)
~.
\end{equation}
This $\TT$-invariant morphism~(\ref{e29}) results a diagram among spaces:
\begin{equation}\label{e32}
\xymatrix{
&
\Bigl(
\displaystyle \int_{\SS^1} \bfX
\Bigr)^{\TT}
\ar[r]^-{\trace^{\TT}}
\ar[d]
&
\End_{\bfX}(\uno)^{\pitchfork \sB\TT}
\ar[d]
\\
\Obj(\bfX)
\ar[r]
\ar[ur]
&
\displaystyle \int_{\SS^1} \bfX
\ar[r]_-{(\ref{e29})}^-{\trace}
&
\End_{\bfX}(\uno) 
~,
&
V
\mapsto 
\Bigl(
\uno
\xra{\eta}
V^\vee \ot V
\simeq
V \ot V^{\vee}
\xra{\epsilon}
\uno
\Bigr)
~,
}
\end{equation}
for which, we show, the bottom horizontal composite evaluates as indicated, where $V^\vee$ is the dual of $V$, and $\eta$ and $\epsilon$ are the respective unit and the counit for this duality.

\smallskip

\subsubsection*{\bf Conjecture of To\"en--Vezzosi.}
As a consequence of the construction of~(\ref{e32}), we prove Conjectures~4.1 and~5.1 of To\"en--Vezzosi~\cite{toen.vezzosi}, which assert that trace maps have canonical $\TT$-invariant refinements for dualizable objects in general symmetric monoidal $(\oo,1)$-categories: 
see {\bf Corollary~\ref{cor.toenvezzosi}}. 
Specifically, we extend the bottom horizontal trace map~(\ref{e32}) along $\Obj(\bfX) \xra{\rm constant} \Obj(\bfX)^{\pitchfork \SS^1} =: L\Obj(\bfX)$:
\begin{equation}\label{e33}
\trace
\colon
L\bfX
 \longrightarrow 
\End_{\bfX}(\uno)
\end{equation}
To\"en and Vezzosi later showed in~\cite{toen.vezzosi2} that their Conjectures~4.1 and~5.1 of~\cite{toen.vezzosi} are consequences of the 1-dimensional cobordism hypothesis; see also \cite{steinebrunner}, which also makes use of the 1-dimensional cobordism hypothesis.
The logic of this work is, however, independent of the cobordism hypothesis: 
we construct the trace maps~(\ref{e32}) and~(\ref{e33}) using factorization homology.
In~\cite{tang1}, we prove the 1-dimensional cobordism hypothesis using methods that are logically independent of this work, though firmly inspired by its approach.

\smallskip

\subsubsection*{\bf Cobordism hypothesis.}
Recall, after Baez--Dolan~\cite{baezdolan} and Hopkins--Lurie~\cite{cobordism}, that the 1-dimensional cobordism hypothesis asserts that a dualizable object $V\in \bfX$ determines a 1-dimensional TQFT, $\sZ_V$.
Identifying the values of $\sZ_V$ on cobordisms that are disjoint unions of disks is a simple matter using dualizability data of $V$ (e.g., its dual $V^\vee$, the unit $\uno \xra{\eta} V^\vee\ot V$, and the counit $V\ot V^\vee \xra{\epsilon}\uno$).  
Furthermore, identifying the value $\sZ_V(\SS^1)\in \End_{\bfX}(\uno)$ can be achieved by witnessing $\SS^1$ as a composition (in $\Bord$) of such disk-cobordisms.  
For instance, witnessing $\SS^1$ as a union of two hemispherical 1-disks determines an identification
\[
(\uno\xra{\sZ_V(\SS^1)}\uno) 
~\simeq~
(\uno \xra{\eta}V^\vee \ot V \simeq V\ot V^\vee \xra{\epsilon} \uno)
~.
\]
In this way, one can easily identify all values of the sought TQFT $\sZ_V$, via \emph{combinatorial presentations} of 1-manifolds as gluings of disjoint unions of 1-disks along boundary 0-spheres.  
The {\bf key difficulty} in proving the 1-dimensional cobordism hypothesis is to verify coherent compatibilities among each value of $\sZ_V$ determined by a combinatorial presentation.
This can be summarized as, for instance, the problem of constructing the requisite action $\Diff^{\fr}(\SS^1) \racts \sZ_V(\SS^1)$.
While a choice of a combinatorial presentation of $\SS^1$ supplies a natural action of a finite cyclic group $\sC_r\racts \sZ_V(\SS^1)$, this action does not obviously extend as an action of the circle group $\TT \simeq \Diff^{\fr}(\SS^1)$.  
Such an action $\TT\racts \sZ_V(\SS^1)$ is a map between spaces,
\begin{equation}\label{f3}
\bigl\lag \TT \racts \sZ_V(\SS^1) \bigr\rag 
\colon
\sB \TT
\xra{~\lag \TT \racts \SS^1 \rag ~}
\End_{\Bord}(\emptyset) 
\xra{~\sZ_V~}
\End_{\bfX}(\uno)
~,
\end{equation}
which is a non-trivial homotopy-theoretic datum.  
For instance, as demonstrated in~(\ref{e28}) after~(\ref{e26}), for any commutative ring $\Bbbk$, it determines an element $c_{2d}\in  \sH_{2d}\bigl( \End_{\bfX}(\uno);\Bbbk \bigr)$ in each even-degree homology group.
Each combinatorial definition of trace, and hence of the value $\sZ_V(\SS^1):=\trace(\id_V) \simeq \epsilon \circ \eta$, only determines the degree-0 element.
{\bf Theorem~\ref{r9}} constructs this map~(\ref{f3}) by supplying diagram~(\ref{e32}), thereby resolving the key difficulty mentioned above.

\smallskip

\subsubsection*{\bf Constructing the $\TT$-invariant trace.}
We now explain how factorization homology is used to prove {\bf Theorem~\ref{r9}}, and supply diagram~(\ref{e32}).
We frame this discussion in terms of the aforementioned key difficulty: to coherently identify each value of $\sZ_V(\SS^1)$ obtained from a combinatorial presentation of $\SS^1$.

A combinatorial presentation $D \rm{~of~}\SS^1$ determines, and is determined by, a non-empty finite subset $\sk_0(D) \subset \SS^1$;
an {\it equivalence} between two combinatorial presentations $D$ and $E$ of $\SS^1$ is an isotopy between finite subsets $\sk_0(D) \sim \sk_0(E)$.
The collection of combinatorial presentations $D \rm{~of~}\SS^1$ thusly organizes as the $\TT$-equivariant moduli space $\Obj\bigl( \bcD(\SS^1) \bigr) := \underset{r > 0} \coprod \conf_r(\SS^1)_{\Sigma_r} \simeq \underset{r >0} \coprod \frac{\TT}{\sC_r}$.  
Consider the \emph{$\oo$-categorical suspension},
\[
\sS\Bigl(
\Obj\bigl( \bcD(\SS^1) \bigr)
\Bigr)
~:=~
\Obj\bigl( \bcD(\SS^1) \bigr)^{\tl}
\underset{\Obj\bigl( \bcD(\SS^1) \bigr)}
\coprod
\Obj\bigl( \bcD(\SS^1) \bigr)^{\tr}
~,
\]
which is the $(\oo,1)$-category resulting from adjoining two cone-points $\pm \infty$ to $\Obj\bigl( \bcD(\SS^1) \bigr)$ with, for each point $(D{\rm ~of~}\SS^1)\in \Obj\bigl( \bcD(\SS^1)$, unique morphisms $-\infty \xra{!} (D{\rm ~of~}\SS^1) \xra{!} +\infty$.
Assigning to each combinatorial presentation $D\rm{~of~}\SS^1$ its associated value of 
$\sZ_V(D{\rm ~of~}\SS^1)$ with its evident $\ZZ \simeq \Omega \frac{\TT}{\sC_{|\sk_0(D)|}}$-invariance defines a functor that carries the cone-points to the symmetric monoidal unit $\uno \in \bfX$:
\begin{equation}\label{e34}
\sZ_V( - {\rm ~of~} \SS^1) \colon
\sS\Bigl(\Obj\bigl( \bcD(\SS^1) \bigr)\Bigr)
\longrightarrow
\cX
~,
\end{equation}
{\small
\[
\Bigl(
-\infty \xra{!}
(D {\rm~of~} \SS^1)
\xra{!}
+\infty
\Bigr)
\longmapsto
\Bigl(
\uno
\xra{~\eta^{\ot \sk_0(D)}~}
(V^\vee \ot V)^{ \ot \sk_0(D) }
\simeq
(V \ot V^\vee)^{ \ot \pi_0( \SS^1 \smallsetminus \sk_0(D) ) }
\xra{~\epsilon^{\ot \pi_0( \SS^1 \smallsetminus \sk_0(D) )}~}
\uno
\Bigr)
~.
\]
}
Here, the central identification cyclically reassociates the iterated tensor product, which is thereafter naturally indexed by the \bit{Poincar\'e dual} of $(D\rm~of~ \SS^1)$ which is the result of exchanging the $0$- and $1$-dimensional strata of $D$.
This functor~(\ref{e34}) is evidently $\TT$-invariant, thusly extending to the $\TT$-coinvariants:
\begin{equation}\label{e51}
\sZ_V( - {\rm ~of~} \SS^1)_{\TT} \colon
\sS\Bigl(\Obj\bigl( \bcD(\SS^1) \bigr)\Bigr)_{\TT}
~\simeq~
\sS\Bigl( \underset{r >0} \coprod \sB\sC_r \Bigr)
\longrightarrow
\cX
~.
\end{equation}
As the values of this functor~(\ref{e51}) on $\pm \infty$ are $\uno \in \cX$, 
composing morphisms from $-\infty$ to $+\infty$ defines a local system,
\begin{equation}\label{e52}
\Bigl\lag
\sC_- \racts \sZ_V( - {\rm ~of~} \SS^1)_{\TT} 
\Bigr \rag
\colon
\underset{r>0} \coprod \sB \sC_r
\longrightarrow
\underset{r>0} \prod
\End_{\bfX}(\uno)
~.
\end{equation}
As we now explain, these actions~(\ref{e52}) can be extended to the sought action~(\ref{f3}) by considering how two combinatorial presentations are related through moves beyond isotopy.

Consider combinatorial moves among combinatorial presentations of $\SS^1$ generated by identifying two cyclically-adjacent elements in $\sk_0(D)$, and by introducing a new cyclically arranged element to $\sk_0(D)$.  
These moves generate morphisms in an $(\oo,1)$-category $\bcD(\SS^1)$ (introduced in~(\ref{f10})), whose moduli space of objects is $\Obj\bigl( \bcD(\SS^1) \bigr)$.  
A first key technical feature of $\bcD(\SS^1)$ is that its classifying space $\sB \bcD(\SS^1)\simeq \ast$ is contractible (see {\bf Corollary~\ref{t22}}). Connectedness of this classifying space implies any two combinatorial presentations of $\SS^1$ are related by combinatorial moves; contractibility implies there is an essentially unique way to relate two.  
This technical feature implies $\sS\bigl( \bcD(\SS^1) \bigr) \simeq \bcD(\SS^1)^{\tl\!\tr}$ is the result of freely adjoining both an initial object and a final object (see {\bf Observation~\ref{t52}}).
So the key difficulty highlighted above is resolved upon constructing a $\TT$-invariant extension of~(\ref{e34}) among $(\oo,1)$-categories:
\begin{equation}\label{e35}
\xymatrix{
\sS\Bigl(\Obj\bigl( \bcD(\SS^1) \bigr)\Bigr)
\ar[rrrrrrrrrr]^-{ 
\uno 
\xra{ 
\eta^{\ot \sk_0(-)}} 
(V^\vee \ot V)^{\ot \sk_0(-)} \simeq (V \ot V^\vee)^{\ot \pi_0( \SS^1 \smallsetminus \sk_0(-) )} 
\xra{\epsilon^{\ot \pi_0( \SS^1 \smallsetminus \sk_0(-) )}}
\uno}
\ar[d]
&&
&&
&&
&&
&&
\bfX
\\
\bcD(\SS^1)^{\tl\!\tr}
\ar@{-->}[urrrrrrrrrr]
&&
&&
&&
&&
&&
~.
}
\end{equation}
Such a factorization determines, via composing morphisms from $-\infty$ to $+\infty$, the bottom horizontal functor
\[
\xymatrix{
\ast
\ar@{-->}[rrrr]^-{\bigl \lag \TT \racts \sZ_V(\SS^1) \bigr\rag}
\ar@{-->}[drrrr]^-{\exists !}_-{\lag  \sZ_V(\SS^1) \rag}
&&
&&
\End_{\bfX}(\uno)^{\pitchfork \sB \TT}
\ar[d]
\\
\bcD(\SS^1)
\ar[u]^-{\rm localization}
\ar[rrrr]_-{{\rm (\ref{e35})'s~filler}}
&&
&&
\End_{\bfX}(\uno)
~.
}
\]
The aforementioned key technical feature of $\bcD(\SS^1)$ that its classifying space is contractible (see {\bf Corollary~\ref{t22}}) grants a unique diagonal filler -- this achieves the construction of a point $\sZ_V(\SS^1) \in \End_{\bfX}(\uno)$ that is manifestly compatible with all combinatorially presented values. 
Furthermore, the $\TT$-invariance of such a factorization in~(\ref{e35}) is a further top horizontal lift -- this achieves the sought canonical action $\TT \racts \sZ_V(\SS^1)$ of~(\ref{f3}).

The role of factorization homology in this paper is to construct an extension~(\ref{e35}).
Specifically, the dualizable object $V\in \bfX$ is selected by a monoidal functor $\Dual \xra{\lag V \text{ (with dual) }\rag}  \bfX$ from the \bit{walking dual} (see Definition~\ref{d17}).
Applying the $\alpha$-version of factorization homology over $\SS^1$ to this \emph{monoidal} functor $\lag V\rag$ supplies functors among $(\infty,1)$-categories,
\begin{equation}\label{e36}
\displaystyle \int^{\alpha}_{\SS^1} \Dual
\xra{~\int_{\SS^1} \lag V \rag~}
\displaystyle \int^{\alpha}_{\SS^1} \bfX
\xra{~\underset{\bfX}\bigotimes~}
\bfX
~,
\end{equation}
in which the second functor exploits that the monoidal structure on $\bfX$ is restricted from its symmetric monoidal structure (see~(\ref{e6})).
{\bf Corollary~\ref{t5}} supplies a canonical functor
\begin{equation}\label{e37}
\bcD(\SS^1)^{\tl\!\tr}
\longrightarrow
\displaystyle \int^{\alpha}_{\SS^1} \Dual
\end{equation}
that carries $\pm \infty$ to $\uno \in \int_{\SS^1}\Dual$.  
Composing~(\ref{e36}) and~(\ref{e37}) results in a composite $\TT$-equivariant functor
\begin{equation}\label{e40}
\bcD(\SS^1)^{\tl\!\tr}
\xra{~(\ref{e37})~}
\displaystyle \int^{\alpha}_{\SS^1} \Dual
\xra{~\int_{\SS^1} \lag V \rag~}
\displaystyle \int^{\alpha}_{\SS^1} \bfX
\xra{~\underset{\bfX}\bigotimes~}
\bfX
~,
\end{equation}
which is the sought $\TT$-invariant factorization~(\ref{e35}).
Now, provided $\bfX$ admits geometric realizations, then the colimit
\[
\displaystyle \int^\alpha_{\SS^1} \un{\End}_{\bfX}(V)
~:=~
\colim
\Bigl(
\bcD(\SS^1)^{\tl}
\xra{~(\ref{e40})_{|\bcD(\SS^1)^{\tl}}~}
\bfX
\Bigr)
\]
exists (see {\bf Theorem~\ref{r9}}).
As such, the functor~(\ref{e40}) is precisely the datum of a $\TT$-invariant \bit{trace} morphism in $\bfX$:
\[
\displaystyle \int^\alpha_{\SS^1} \un{\End}_{\bfX}(V)
\xra{~\trace~}
\uno
~.
\]
Furthermore, the endomorphism $\sZ_V(\SS^1)$ constructed above is the composition
\[
\sZ_V(\SS^1)
\colon
\uno
\xra{~\unit~}
{\sf HH}\bigl(
\un{\End}_{\bfX}(\uno)
\bigr)
\xra{~\trace~}
\uno
~,
\]
which makes no reference to a combinatorial presentation of $\SS^1$, and it is manifestly $\TT$-invariant (see {\bf Theorem~\ref{r9}}).

\subsection*{Implementation of factorization homology}\label{sec.fact.versions}
This work uses the theory of factorization homology, the development of which is distributed across multiple works.  
Here, we guide a reader through these works, tailored to how we use factorization homology in this paper.
\begin{itemize}

\item
The work~\cite{oldfact} establishes factorization homology of $\cE_n$-algebras over framed $n$-manifolds.
Let us specialize to dimension $n=1$, since that is what is relevant for the present paper.
There is a symmetric monoidal $(\oo,1)$-category $\Mfld_1^{\fr}$ in which an object is a framed $1$-manifold, the space of morphisms is that of framed embeddings between them, and the symmetric monoidal structure is given by disjoint union of framed $1$-manifolds. Fix a symmetric monoidal $(\oo,1)$-category $\cV$ that admits sifted colimits and such that $\ot$ distributes over sifted colimits.  
Consider the $(\oo,1)$-category $\Alg(\cV)$ of associative algebras in $\cV$.
Factorization homology, as established in~\cite{oldfact} in the case $n=1$, is a functor
\[
\int
\colon
\Alg(\cV)
\longrightarrow
\Fun^{\ot}( \Mfld_1^{\fr} , \cV)
~,\qquad
A
\longmapsto
\left(
M
\longmapsto
\int_M A
\right)
~.
\]
One may reasonably refer to this notion of factorization homology as the $\alpha$-version, since it appears first.  
Therefore, to disambiguate, we use the notation
\[
\int^\alpha_M A
\]
for the version of factorization homology established in~\cite{oldfact}.

Informally, for $A$ an associative algebra and for $M$ a framed 1-manifold, a point in $\int^\alpha_M A$ is a pair $(S, \ell)$ consisting of a finite subset $S \subset M$ the complement of which is a disjoint union of open 1-disks, and a labeling $\ell$ of each such disk by an element in $A$.
\\

\item
The work~\cite{fact1} establishes factorization homology of $(\infty,n)$-categories over vari-framed $n$-manifolds (at least for $n\leq 2$ as explained in~\cite{corrigendum}).  
Let us specialize to dimension $n=1$, since that is what is relevant for the present paper.
There is an $(\oo,1)$-category $\cMfd_1^{\sfr}$ informally described as follows (see~\cite{fact1} for full details).  
An object is a \bit{solidly 1-framed stratified space}, which is a compact stratified space $X$ equipped with an injection $\tau_X \overset{\varphi}\hookrightarrow \epsilon^1_X$ of the tangent constructible bundle of $X$ into a trivial rank-1 vector bundle over $X$.  The existence of such an injection $\varphi$ implies each stratum of $X$ has dimension at most 1 -- so $X$ is, in particular, a CW complex of dimension at most 1.  Framed 1-manifolds are examples of solidly 1-framed stratified spaces. Consider the $(\oo,1)$-category $\Cat_1$ of $(\infty,1)$-categories.

Factorization homology, as established in~\cite{fact1}, is a functor
\[
\int
\colon
\Cat_1
\longrightarrow
\Fun(\Mfd_1^{\sfr} , \Spaces)
~,\qquad
\cC
\longmapsto
\left(
M
\longmapsto
\int_M \cC
\right)
~.
\]
One might reasonably refer to this notion of factorization homology as the $\beta$-version, or geometric $\beta$-version, since it appears second and is founded on differential topology.
However, because this version subsumes the previous (as discussed in the next point), we will not use the term $\beta$.

Informally, for $\cC$ an $(\infty,1)$-category and for $M$ a framed 1-manifold, a point in $\int^\beta_M \cC$ is a pair $(S, \ell)$ consisting of a finite subset $S \subset M$, the complement of which is a disjoint union of open 1-disks, and a labeling $\ell$ of each such disk by a morphism in $\cC$ and each element in $S$ by an object in $\cC$ with source-target compatibility.
Note that such a $(S \subset M)$ determines a finite directed graph (i.e., a quiver), and that $\ell$ is a representation of that quiver in $\cC$.
\\

\item
Taking deloop of monoids (i.e., associative algebras in $\Spaces$) defines a functor $\fB \colon \Alg(\Spaces) \to \Cat_1$.
Proposition~\ref{t23} articulates an agreement between the $\alpha$- and the $\beta$-versions of factorization homology of a monoid: $\int^\alpha_M A \simeq \int^\beta_M \fB A$. 
\\

\item
The work~\cite{circle} establishes factorization homology of $(\infty,1)$-categories over quivers and other objects formally derived by such.  
Specifically, fix an $(\oo,1)$-category $\cX$ that admits finite limits, sifted colimits, and such that sifted colimits distribute over finite products.  
($\cX = \Spaces$ and $\cX= \Cat_1$ are examples.)
Consider the $(\oo,1)$-category $\fCat_1[\cX]$ of category-objects in $\cX$ -- these are simplicial objects in $\cX$ that satisfy a Segal condition.
Consider the category $\Quiv$ of quivers (i.e., $(\infty,1)$-categories freely generated by finite directed graphs).
The standard functor $\bDelta \to \Quiv$ selecting the linear quivers can be viewed as a category-object in $\Quiv^{\op}$
There is an $(\oo,1)$-category $\bcM$ with finite products, equipped with a finite-product preserving functor $\Quiv^{\op} \to \bcM$, that is initial among those for which the Hochschild homology of $\bDelta^{\op} \to \Quiv^{\op} \to \bcM$ exists.
Factorization homology, as established in~\cite{circle}, is a functor
\[
\int
\colon
\fCat_1[\cX]
\longrightarrow
\Fun(\bcM, \cX)
~,\qquad
\cC
\longmapsto
\left(
M
\longmapsto
\int_M \cC
\right)
~.
\]
One might reasonably refer to this notion of factorization homology as the combinatorial $\beta$-version, since it is founded on the combinatorics of quivers.
The work~\cite{circle} is logically independent of~\cite{oldfact} and of~\cite{fact1}.  
Proposition~B.3.5 of~\cite{circle} identifies the $(\oo,1)$-categories $\bcM \simeq \cMfd_1^{\sfr}$;
Theorem~B.3.7 of~\cite{circle} states that, through this identification, the combinatorial $\beta$-version of factorization homology established in~\cite{circle} agrees with the geometric $\beta$-version of factorization homology established in~\cite{fact1}.
One may therefore reasonably referred to either of these notions of factorization homology simply as the $\beta$-version, or simply as factorization homology.

Informally, for $\cC$ an $(\infty,1)$-category and for $M \in \bcM$, a point in $\int^\beta_M \cC$ is represented by a quiver $\Gamma$ refining $M$ and a representation of $\Gamma$ in $\cC$.
\\

\end{itemize}

The present paper makes use the $\beta$-version of factorization homology, as well as the $\alpha$-version.
Because we find the geometric approach to factorization homology more intuitive, in this paper we will embrace the developments in~\cite{fact1}, though a reader is invited to translate each reference to this geometric approach to the combinatorial approach of quivers as in~\cite{circle}.

\begin{remark}
The work~\cite{enriched1} constructs factorization homology of enriched $(\infty,1)$-categories.
That work is especially relevant for topological quantum field theory.
That work builds upon those mentioned above, though the present paper does not make use of it.
\end{remark}

\subsection*{Index of notation}

We employ the following notational conventions in this paper.

\begin{itemize}

\item $\Spaces$ is the $(\oo,1)$-category of spaces (i.e., the $\oo$-categorical localization of the category of Kan complexes, localized on homotopy equivalences).

\item For $n \geq 0$, the $(\oo,1)$-category $\Cat_n$ is that of $(\oo,n)$-categories, where it is understood that $\Cat_0 = \Spaces$.  
For $0\leq k \leq n$, there is an understood fully-faithful functor $\Cat_k \subset \Cat_n$.

\item
Let $X$ be a space, regarded as an $\infty$-groupoid.
Let $Z\in \cK$ be an object in an $(\oo,1)$-category
The \bit{cotensor} of $Z$ by $X$ is the limit $Z^{\pitchfork X} := \limit( X \xra{!} \ast \xra{\lag Z \rag} \cK )$ of the constant functor at $Z$ from $X$.

\item
$\TT=\sB \ZZ$ is the \bit{circle group}.
For $X\in \cK$, the object in $\cK$ of \bit{$\TT$-modules} in $X$ is $X^{\pitchfork \sB \TT}$.
So, for $\cC\in \Cat_n$, the $(\infty,n)$-category of $\TT$-modules in $\cC$ is $\Fun(\sB \TT , \cC) \simeq \cC^{\pitchfork \sB \TT}$.  

\item
For $\cK$ an $(\oo,1)$-category, the $(\oo,1)$-category $\cK^{\tl}$ is the result of freely adjoining an initial object to $\cK$, and the $(\oo,1)$-category $\cK^{\tr}$ is the result of freely adjoining a final object to $\cK$.
The $(\oo,1)$-category $\cK^{\tl\!\tr}$ is $(\cK^{\tl})^{\tr}\simeq (\cK^{\tr})^{\tl}$.

\item\label{d0}
Let ${\bfX}$ be a symmetric monoidal $(\infty,1)$-category.
The object $\uno \in \bfX$ is its symmetric monoidal unit.
The category-object internal to $(\infty,1)$-categories $\fB\bfX$ is the deloop of the underlying monoidal $(\oo,1)$-category $\bfX$.
Specifically, as a simplicial-category, $\fB \bfX = \bBar_\bullet(\bfX)$ is the bar-construction.

\end{itemize}

\section{Factorization homology of the walking adjunction}

The main technical result of this section is the calculation of $\int_{\SS^1}\Adj$, the factorization homology over the circle of the walking adjunction. 
We state this calculation, Theorem~\ref{t10}, below; we prove it at the conclusion of~\S\ref{sec.fact.adj}. 
See Appendix~\S\ref{sec.adjunction} for definitions of 
the paracyclic category $\para$, the walking dual $\Dual$, and the walking adjunction $\Adj$; see Appendix~\S\ref{sec.recall.fact} for a survey of the $\alpha$- and $\beta$-versions of factorization homology.

\begin{theorem}\label{t10}
There is a canonical equivalence between cospans of $\TT\simeq \sB\ZZ$-module-$(\oo,1)$-categories:
\[
\Bigl(
\int_{\SS^1}^\alpha\un{\End}_{\Adj}(-)
\longrightarrow
\int_{\SS^1} \Adj
\longleftarrow
\int_{\SS^1}^\alpha \un{\End}_{\Adj}(+)
\Bigr)
~\simeq~
\Bigl(
\para^{\tl}
\longrightarrow
\para^{\tl\!\tr}
\longleftarrow
(\para^{\op})^{\tr}
\Bigr)
~.
\]
In other words, there are canonical $\TT\simeq \sB\ZZ$-equivariant equivalences
\begin{eqnarray}
\nonumber
\int_{\SS^1}\Adj
&~\simeq~&
\para^{\tl\!\tr}
\\
\nonumber
\int_{\SS^1}^\alpha \un{\End}_{\Adj}(-)
&~\simeq~&
\para^{\tl}
\\
\nonumber
\int_{\SS^1}^\alpha \un{\End}_{\Adj}(+)
&~\simeq~&
(\para^{\op})^{\tr}
~,
\end{eqnarray}
compatible with respect to the contravariant automorphism $\para^{\op}\simeq \para$ of Observation~\ref{t34}.
\\
\end{theorem}

The next result makes reference to the monoidal $(\infty,1)$-category $\Dual$ of Definition~\ref{d17}, as well as the $\TT$-module-$(\infty,1)$-category $\bcD(\SS^1)$ of~(\ref{f10}).
\begin{cor}\label{t5}
There is a canonical $\TT$-equivariant functor between $(\infty,1)$-categories,
\[
\bcD(\SS^1)^{\tl\!\tr}
\longrightarrow
\displaystyle \int^\alpha_{\SS^1} \Dual
~,
\]
that carries both cone-points to $\uno \in \int^\alpha_{\SS^1} \Dual$.

\end{cor}

\begin{proof}
The $\TT$-equivariant functor is the composition of the following $\TT$-equivariant functors, explained below.
We explain the following sequence of $\TT$-equivariant functors, the composite of which is that of the corollary:
\begin{eqnarray}
\nonumber
\bcD(\SS^1)^{\tl\!\tr}
&
\underset{\rm Cor~\ref{t21}}{\xra{~\simeq~}}
&
\para^{\tl\!\tr}
\\
\nonumber
&
\underset{\rm Thm~\ref{t10}}{~\simeq~}
&
\displaystyle \int_{\SS^1} \Adj
\\
\nonumber
&
\overset{\rm \int_{\SS^1} (\ref{e15})}
\longrightarrow
&
\displaystyle \int_{\SS^1} \fB\Dual
\\
\nonumber
&
\underset{\rm Prop~\ref{t23}}{~\simeq~}
&
\displaystyle \int^\alpha_{\SS^1} \Dual
~.
\end{eqnarray}
The first functor is Corollary~\ref{t21}, which identifies $\bcD(\SS^1)$ and $\para$.  
The second functor is Theorem~\ref{t10}.
The third functor is factorization homology applied to the quotient functor~(\ref{e15})
The fourth functor is Proposition~\ref{t23}, which identifies the $\alpha$-version of factorization homology with the $\beta$-version of factorization homology.
\\

\end{proof}

\subsection{Factorization homology of the walking monad and comonad}

See Appendix~\S\ref{sec.adjunction} for definition of the walking monad $\sO$. The goal of this section is to identify the values of $\alpha$-factorization homology of the monoidal 
$(\infty,1)$-categories $\sO$ and $\sO^{\op}$.
These values will be $(\infty,1)$-categories.

\begin{lemma}\label{t31}
For each framed 1-manifold $M$, consider the composite functor among $(\infty,1)$-categories:
\[
\Disk^{\fr}_{1/M} \xra{\rm forget} \Disk^{\fr}_1 \xra{~\pi_0~} {\sf Sets}
~.
\]
\begin{enumerate}
\item
In the case that $M=\RR^1$, the linear order on $\pi_0(U)$ inherited from an embedding $U\hookrightarrow \RR^1$ determines a lift of this composite functor to $\sO\simeq (\bDelta)^{\tl}$.  This lift is an equivalence between $(\infty,1)$-categories:
\begin{equation}\label{e44}
\Disk^{\fr}_{1/\RR^1} \xra{~\simeq~} \sO \simeq (\bDelta)^{\tl}
~,\qquad
(U\hookrightarrow \RR^1)
\mapsto 
\pi_0(U)
~.
\end{equation}

\item
For the case that $M=\SS^1$, consider the universal covering space $\RR^1 \xra{\sf exp} \SS^1$.
The linear order on $\pi_0({\sf exp}^{-1} U )$ inherited from an embedding ${\sf exp}^{-1}(U)\hookrightarrow \RR^1$, together with the $\pi_1(\SS^1)=\ZZ$-action by deck-transformations, determines a lift of this composite functor to $\para^{\tl}$.  
This lift is an equivalence between $\TT\simeq \sB\ZZ$-module-$(\infty,1)$-categories:
\begin{equation}\label{e45}
\Disk^{\fr}_{1/\SS^1} \xra{~\simeq~} \para^{\tl}
~,\qquad
(U\hookrightarrow \SS^1)
\mapsto 
\pi_0({\sf exp}^{-1}U)
~.
\end{equation}

\end{enumerate}

\end{lemma}

\begin{proof}
Let $M$ be a framed 1-manifold. Lemma~2.12 of~\cite{oldfact} identifies the space of objects of the $\infty$-overcategory as the space of unordered configurations of finite subsets of $M$:
\[
\Obj(\Disk^{\fr}_{1/M})
~\simeq~ 
\underset{r \geq 0} \coprod \conf_r(M)_{\Sigma_r} 
~.
\]
For each $r\geq 0$, straight-line homotopy in $\RR^1$ defines a $\Sigma_r$-equivariant deformation retraction of $\conf_r(\RR^1)$ onto the subspace of evenly spaced configurations in $[0,r] \subset \RR$ that contain $0$ and $r$.  
This subspace has one connected component for each linear ordering of the set $\{1,\dots,r\}$, and each connected component is a singleton.   
Consequently, the space $\conf_r(\RR^1)_{\Sigma_r} \simeq \ast$ is contractible.
We conclude that the functor~(\ref{e44}) induces an equivalence between spaces of objects:
\begin{equation}
\label{e44'}
\Obj(\Disk^{\fr}_{1/\RR^1})
\xra{~\simeq~}
\Obj(\sO)
~.
\end{equation}
In particular~(\ref{e44}) is essentially surjective.

For $(U \hookrightarrow M) \in \Disk^{\fr}_{1/M}$ an object, there is a canonical identification of the $\infty$-overcategory,
\[
(\Disk^{\fr}_{1/M})_{/(U \hookrightarrow M)}
~\simeq~
\Disk^{\fr}_{1/U}
~.
\]
Taking disjoint unions defines a functor from the $\pi_0 U$-fold product of $\infty$-overcategories,
\[
(\Disk^{\fr}_{1/\RR^1})^{\times \pi_0 U}
\xra{~\sqcup~}
\Disk^{\fr}_{1/U}
~.
\]
This functor is an equivalence, with inverse given $(V \hookrightarrow U) \mapsto \left( V \cap U_\alpha \hookrightarrow U_\alpha \right)_{\alpha \in \pi_0 U}$.
Observation~\ref{t99} of \S\ref{sec.monad} supplies a similar equivalence concerning the category $\sO$.
The equivalence~(\ref{e44'}) then implies that the functor~(\ref{e44}) induces an equivalence between spaces of objects of over categories:
\[
\Obj\left(
(\Disk^{\fr}_{1/\RR^1})_{/(U \hookrightarrow \RR^1)} 
\right)
~\simeq~
\Obj(\Disk^{\fr}_{1/\RR^1})^{\times \pi_0 U}
\underset{(\ref{e44'})}{\longrightarrow}
\Obj(\sO)^{\times \pi_0 U}
\underset{\rm Obs~\ref{t99}}{~\simeq~}
\Obj(\sO_{/\pi_0 U})
~.
\]
This equivalence evidently lies over the equivalence~(\ref{e44'}), via the forgetful functors $(\Disk^{\fr}_{1/\RR^1})_{/(U \hookrightarrow \RR^1)} \to \Disk^{\fr}_{1/\RR^1}$ and $\sO_{/\pi_0 U} \to \sO$.  
Now, generally, for $c,c'\in \cC$ objects in an $(\oo,1)$-category, the fiber of $\cC_{/c'} \to \cC$ over $c$ is the space $\Hom_{\cC}(c',c)$.
We conclude that the functor~(\ref{e44}) is fully faithful.

Similarly, for each $r> 0$, straight-line homotopy in the universal cover $\RR^1$ of $\SS^1$ defines a $\Sigma_r$-equivariant deformation retraction of $\conf_r(\SS^1)$ onto its subspace of evenly spaced configurations.  This subspace has one connected component for each cyclic ordering of $\{1,\dots,r\}$.  
Rotation defines a transitive action by the circle group $\SS^1$ on each connected component of this subspace, whose stabilizer of a point is $\sC_r \subset \SS^1$, the cyclic group of order $r$.
In this way, we have a homotopy equivalence $\conf_r(\SS^1)_{\Sigma_r} \simeq \SS^1_{/\sC_r} \simeq \SS^1 \simeq \sB \ZZ$.  
We conclude that the functor~(\ref{e45}) induces an equivalence between spaces of objects:
\begin{equation}
\label{e45'}
\Obj(\Disk^{\fr}_{1/\SS^1})
\xra{~\simeq~}
\Obj(\para^{\tl})
~.
\end{equation}
In particular~(\ref{e45}) is essentially surjective.

Now, for each $(\ZZ \lacts I) \in \para^{\tl}$, with quotient $I \xra{q} I_{/\ZZ}$, taking joins defines a functor:
\[
\sO^{\times I_{/\ZZ}}
\longrightarrow
(\para^{\tl})_{/I}
~,\qquad
(J_k)_{k\in I_{/\ZZ}}
\longmapsto 
\left(
\underset{i \in I} \bigstar J_{q(i)}
\to 
\underset{i \in I} \bigstar \ast
= I
\right)
~.
\]
This functor is an equivalence, with inverse given by $(J\xra{f} I)\mapsto \left(  f^{-1}(i)) \right)_{i\in I_0}$, where $I_0 \subset I$ is a choice of fundamental domain of $(\ZZ \lacts I)$.
The equivalence~(\ref{e44'}) then implies that the functor~(\ref{e45}) induces an equivalence between spaces of objects of overcategories:
\[
\Obj\left(
(\Disk^{\fr}_{1/\SS^1})_{/(U \hookrightarrow \SS^1)} 
\right)
~\simeq~
\Obj(\Disk^{\fr}_{1/\RR^1})^{\times \pi_0 U}
\underset{(\ref{e44'})}{\longrightarrow}
\Obj(\sO)^{\times \pi_0 U}
~\simeq~
\Obj\left(
(\para^{\tl})_{/\pi_0 U} \right)
~.
\]
This equivalence evidently lies over the equivalence~(\ref{e45'}), via the forgetful functors $(\Disk^{\fr}_{1/\SS^1})_{/(U \hookrightarrow \SS^1)} \to \Disk^{\fr}_{1/\SS^1}$ and $(\para^{\tl})_{/\pi_0 U} \to \para^{\tl}$.  
We conclude via the same logic as above that the functor~(\ref{e45}) is an equivalence, as desired.

\end{proof}

\begin{lemma}\label{t11}
For each compact solidly 1-framed stratified space $M$, there are canonical equivalences between $(\infty,1)$-categories:
\[
\int_M \fB \un{\End}_{\Adj}(-)
~\simeq~
\Disk^{\fr}_{1/M\smallsetminus \sk_0(M)}
\qquad
\text{ and }
\qquad
\int_M  \fB \un{\End}_{\Adj}(+)
~\simeq~
\bigl(\Disk^{\fr}_{1/M\smallsetminus \sk_0(M)}\bigr)^{\op}
~.
\]

\end{lemma}

\begin{proof}
Theorem~\ref{t7} gives equivalences $\un{\End}_{\Adj}(+)^{\op} \simeq \sO \simeq \un{\End}_{\Adj}(-)$ of monoidal $(\infty,1)$-categories.
Therefore, by taking opposites of categories, each of the equivalences in the statement of this lemma implies the other.
We are therefore reduced to establishing a canonical equivalence between $(\infty,1)$-categories: 
$\int_M  \fB \sO
\simeq
\Disk^{\fr}_{1/M\smallsetminus \sk_0(M)}
$.
Through Proposition~\ref{t23}, which identifies $\alpha$-factorization homology as an instance of $\beta$-factorization homology, it is sufficient to show that, for all (finitary) framed 1-manifolds $M'\in \Mfld_1^{\fr}$, there is a canonical equivalence between $(\infty,1)$-categories: 
\[
\int_{M'}^\alpha \sO
~\simeq~ 
\Disk^{\fr}_{1/M'}
~.
\]

Consider the functor
\begin{equation}\label{e43}
\Disk^{\fr}_{1/-}\colon \Mfld_1^{\fr}
\longrightarrow
\Cat_1
~,\qquad
M'
\mapsto 
\Disk^{\fr}_{1/M'}
~.
\end{equation}
The above equivalence is implied by establishing an equivalence between functors from $\Mfld_1^{\fr}$ to $\Cat_1$:
\[
\int_-^{\alpha} \sO
~\simeq~
\Disk^{\fr}_{1/-}
~.
\]
Lemma~\ref{t31}(1) gives that these two functors evaluate identically on $\RR^1$.

For each finite-fold disjoint union $\underset{j\in J} \bigsqcup M'_j$ of finitary framed 1-manifolds, the functor
\[
\Disk^{\fr}_{1/\underset{j\in J} \bigsqcup M'_j}
\xra{~\simeq~}
\prod_{j\in J} \Disk^{\fr}_{1/M'_j}
~,\qquad
\bigl(
U \hookrightarrow \underset{j\in J} \bigsqcup M'_j
\bigr)
\mapsto
\Bigl(
j\mapsto
\bigl(
U\cap M'_j \hookrightarrow M'_j
\bigr)
\Bigr)
~,
\]
is an equivalence.  
In this way, the functor~(\ref{e43}) extends as a symmetric monoidal functor, where the symmetric monoidal structure on the domain is disjoint union, and that on the codomain is product.  
The equivalence~(\ref{e44}) thereby extends as an equivalence between monoidal $(\infty,1)$-categories.  
Lemma~2.4.1
in~\cite{zp} verifies that this symmetric monoidal functor~(\ref{e43}) satisfies the $\ot$-excision condition.  
Theorem~1.2 of~\cite{oldfact} gives that the canonical natural transformation
\[
\int_{-}^{\alpha}
\sO
~
\underset{(\ref{e44})}{~\simeq~}
~
\int_{-}^{\alpha}
\Disk^{\fr}_{1/\RR^1}
\xra{~\simeq~}
\Disk^{\fr}_{1/-}
\]
is an equivalence.

\end{proof}

After Lemma~\ref{t11}, Lemma~\ref{t31} gives the following.
\begin{cor}\label{t30}
There are canonical equivalences between $(\infty,1)$-categories:
\[
\int_{\RR^1}^\alpha\un{\End}_{\Adj}(-)
\xra{~\simeq~}
\bDelta^{\tl}
\qquad\text{ and }\qquad
\int_{\RR^1}^\alpha \un{\End}_{\Adj}(+)
\xra{~\simeq~}
(\bdelta^{\op})^{\tr}
~.
\]
There are canonical equivalences of $\TT\simeq \sB\ZZ$-module $(\infty,1)$-categories:
\[
\int_{\SS^1}^\alpha\un{\End}_{\Adj}(-)
\xra{~\simeq~}
\para^{\tl}
\qquad\text{ and }\qquad
\int_{\SS^1}^\alpha \un{\End}_{\Adj}(+)
\xra{~\simeq~}
(\para^{\op})^{\tr}
~.
\\
\]
\end{cor}

\subsection{Factorization homology of the walking adjunction}\label{sec.fact.adj}

Note the standard fully-faithful embedding
\[
\Cat_2
~\subset~
\fCat_1[\Cat_1]
~:=~
\Fun^{\sf Segal}( \bDelta^{\op} , \Cat_1 )
\]
into the $(\oo,1)$-category of those simplicial $(\infty,1)$-categories that carry (the opposites of) Segal covering diagrams to limit diagrams.  
Through this embedding, factorization homology (see~\S\ref{sec.fact.beta}) defines a functor
\[
\int\colon \Cat_2~\subset~\fCat_1[\Cat_1]
\longrightarrow
\Fun\bigl( \bcM , \Cat_1 \bigr)
~,\qquad
\cC
\mapsto 
\Bigl(
M\mapsto \int_M \cC
\Bigr)
~.
\]
By Lemma~\ref{t24}, for each $\cC \in \Cat_2$, the value of $\int \cC$ over $M\in \bcM$ is equivalent to the colimit in $\Cat_1$ of the functor:
\begin{equation}\label{e59}
(\ref{e38})\colon \bcD(M) \xra{~\rm forget~}\bcD \xra{~\Hom_{\Cat_2}\bigl (\fC(-) , \cC \bigr)~} \Cat_1
~.
\end{equation}

\begin{notation}\label{d18}
Let $\cC\in \Cat_2$ be an $(\infty,2)$-category.  
For each compact solidly 1-framed stratified space $M\in \bcM$, the coCartesian fibration
\begin{equation}\label{e39}
\int^{\lax}_M  \cC
\longrightarrow
\bcD(M)
\end{equation}
is the unstraightening of the functor~(\ref{e59}).  

\end{notation}

\begin{observation}\label{t46}
Proposition~\ref{t48} implies, for each $M\in \bcM$, that the $(\oo,1)$-category $\bcD(M)$ is an ordinary category: the spaces of morphisms between any two objects is a 0-type.
Consequently, if $\cC$ is ordinary 2-category, then $\int^{\lax}_M  \cC$ is an ordinary category.  
In particular, $\int^{\lax}_M  \Adj$, and any full $\infty$-subcategory of it, is an ordinary category. 
\end{observation}

\begin{remark}\label{r7}
We can explicitly describe the $(\infty,1)$-category introduced in Notation~\ref{d18}.  
Let $\cC\in \Cat_2$ be an $(\infty,2)$-category.  
Let $M\in \bcM$ be a compact solidly 1-framed stratified space.
\begin{itemize}
\item
An object in the $(\infty,1)$-category $\int^{\lax}_M  \cC$ is a pair $\bigl(D\xra{\sf ref} M , \fC(D)\xra{\ell} \cC \bigr)$ consisting of a disk-refinement of $M$, together with a functor to $\cC$ from the free $(\infty,1)$-category on the underlying finite directed graph $D$.  
(See Observation~\ref{t57} for a simple way to record such an $\ell$.)

\item
A morphism in the $(\infty,1)$-category $\int^{\lax}_M \cC$, from $(D_s,\ell_s)$ to $(D_t, \ell_t)$, is a morphism $D_s \xra{X} D_t$ in $\bcD(M)$, together with a 2-cell making the diagram among $(\infty,2)$-categories 
{\Small
\begin{equation}\label{e62}
\xymatrix{
\fC(D_s)
\ar[dr]_-{\ell_s}
&
\overset{\gamma}\Rightarrow
&
\fC(D_t) 
\ar[dl]^-{\ell_t}
\ar@(u,u)[ll]_-{\fC(X)}
\\
&
\cC
&
}
\end{equation}
}
commute.
Specifically, such a 2-cell is a morphism between $(\infty,2)$-categories $\fC(D_t) \wr c_1 \xra{\gamma} \cC$ extending $\fC(D_t) \wr \partial c_1 \xra{ \lag~ \ell_s \circ \fC(X)~  ,~ \ell_t~\rag } \cC$.

\item
Composition in $\int^{\lax}_M\cC$ is implemented by concatenating such diagrams~(\ref{e62}) horizontally, and composing morphisms in $\bcD(M)$ while composing such 2-cells.

\item
A morphism in the $(\infty,1)$-category $\int^{\lax}_M \cC$, depicted say as~(\ref{e62}), is coCartesian over $\bcD(M)$ if and only if the 2-cell $\gamma$ in the diagram is invertible, which is to say it is an identification $\ell_s \circ \fC(X) \simeq \ell_t$ between functors from $\fC(D_t)$ to $\cC$.
In particular, such a morphism is an equivalence if and only if the morphism $X$ in $\bcD(M)$ is an equivalence and the 2-cell $\gamma$ is invertible.

\end{itemize}

\end{remark}

\begin{observation}\label{t39}
For each $(\infty,2)$-category $\cC$, and each compact solidly 1-framed stratified space $M\in \bcM$, there is a canonical functor between $(\infty,1)$-categories
\[
\int^{\lax}_M \cC
\xra{~\rm localization~}
\colim\Bigl(
\bcD(M)
\xra{~(\ref{e59})~}
\Cat_1
\Bigr)
\underset{\rm Lem~\ref{t24}}{\xra{~\simeq~}}
\int_M \cC
~,
\]
which exhibits a localization on the coCartesian morphisms of~(\ref{e39}):
\[
\Bigl(
\int^{\lax}_M \cC 
\Bigr)
\Bigl[ {{\sf cCart}_{/\bcD(M)}}^{-1} \Bigr]
\xra{~\simeq~}
\int_M \cC
~.
\]

\end{observation}

To prove Theorem~\ref{t10}, we analyze the coCartesian fibration $\int^{\lax}_{\SS^1} \Adj \to \bcD(\SS^1)$.  
Consider the diagram of $(\infty,2)$-categories and fully-faithful functors among them:
\begin{equation}\label{e42}
\ast
\overset{ \rm unit }{~\hookrightarrow~}
\fB \sO
\underset{\rm Thm~\ref{t7}}\simeq
\fB \un{\End}_{\Adj}(-)
~\hookrightarrow~
\Adj
~\hookleftarrow~
\fB \un{\End}_{\Adj}(+)
\underset{\rm Thm~\ref{t7}}\simeq
\fB \sO^{\op}
\overset{ \rm counit }{~\hookleftarrow~}
\ast
~.
\end{equation}
Applying $\int^{\lax}_{\SS^1}$ to this diagram~(\ref{e42}) results in a diagram of coCartesian fibrations over $\bcD(\SS^1)$ and fully-faithful functors among them each that preserves coCartesian morphisms over $\bcD(\SS^1)$:
\begin{equation}\label{e61}
\bcD(\SS^1)
\xra{ \int^{\lax}_{\SS^1} {\sf unit }}
\int^{\lax}_{\SS^1} \fB \un{\End}_{\Adj}(-)
~\hookrightarrow~
\int^{\lax}_{\SS^1} \Adj
~\hookleftarrow~
\int^{\lax}_{\SS^1} \fB \un{\End}_{\Adj}(+)
\xla{ \int^{\lax}_{\SS^1} {\sf counit }}
\bcD(\SS^1)
~.
\end{equation}

\begin{observation}\label{t41}
Let $(D\xra{\sf ref} \SS^1 , \fC(D)\xra{\ell} \Adj)\in \int^{\lax}_{\SS^1} \Adj$ be an object.
\begin{enumerate}
\item[$\bf +:$]
This object lies in $\int^{\lax}_{\SS^1} \fB \un{\End}_{\Adj}(+)$ if and only if the map between finite sets $\ell_0\colon \sk_0(D) = \Obj\bigl(\fC(D)\bigr) \xra{\Obj(\ell)} \Obj(\Adj)  = \{\pm\}$ is constant at $+$.

\item[$\bf -:$]
This object lies in $\int^{\lax}_{\SS^1} \fB \un{\End}_{\Adj}(-)$ if and only if the map between finite sets $\ell_0\colon \sk_0(D) = \Obj\bigl(\fC(D)\bigr) \xra{\Obj(\ell)} \Obj(\Adj)  = \{\pm\}$ is constant at $-$.

\item[$\bf unit:$]
This object lies in the image of the fully-faithful functor $\bigl( \int^{\lax}_{\SS^1} {\sf unit }\bigr) $ if and only if the functor $\fC(D) \xra{\ell} \Adj$ is constant at $-$.

\item[$\bf counit:$]
This object lies in the image of the fully-faithful functor $\bigl(\int^{\lax}_{\SS^1} {\sf counit }\bigr)$ if and only if the functor $\fC(D) \xra{\ell} \Adj$ is constant at $+$.

\end{enumerate}
In particular, the canonical functor $\int^{\lax}_{\SS^1} \fB \un{\End}_{\Adj}(-) \amalg \int^{\lax}_{\SS^1} \fB \un{\End}_{\Adj}(+) \longrightarrow \int^{\lax}_{\SS^1} \Adj$ is fully-faithful, which is to say the images of the two fully-faithful cofactors are disjoint.  

\end{observation}

The description of morphisms in $\int^{\lax}_{\SS^1} \cC$ given in Remark~\ref{r7} implies the following.
\begin{observation}\label{t40}
Each of the following fully-faithful functors in~(\ref{e61}), as well as the functor
\[
\int^{\lax}_{\SS^1} \fB \un{\End}_{\Adj}(+)
\amalg
\int^{\lax}_{\SS^1} \fB \un{\End}_{\Adj}(-)
~\hookrightarrow~
\int^{\lax}_{\SS^1} \Adj~,
\]
is a left fibration.
In other words, a morphism in $\int^{\lax}_{\SS^1} \Adj$ belongs to one of these full $\oo$-subcategories if and only if its source does.

\end{observation}

\begin{notation}\label{d19}
Denote the $(\infty,1)$-categories over $\bcD(\SS^1)$:
{\Small
\begin{eqnarray}
\nonumber
\cA~:=~ \int^{\lax}_{\SS^1} \Adj
&
,
&
\cA_+~:=~ \int^{\lax}_{\SS^1} \Adj \smallsetminus \int^{\lax}_{\SS^1} \fB \un{\End}_{\Adj}(-)
~,
\\
\nonumber
\cA_- ~:=~ \int^{\lax}_{\SS^1} \Adj \smallsetminus \int^{\lax}_{\SS^1} \fB \un{\End}_{\Adj}(+)
&
,
&
\cA_0~:= \int^{\lax}_{\SS^1} \Adj \smallsetminus \int^{\lax}_{\SS^1} \fB \un{\End}_{\Adj}(+) \amalg \int^{\lax}_{\SS^1} \fB \un{\End}_{\Adj}(-)
~.
\end{eqnarray}
}
Denote the respective $(\infty,1)$-subcategories, each that consists of the same objects and those morphisms that are carried into $\cA$ as coCartesian morphisms over $\bcD(\SS^1)$:
\[
{\cA_0}^{\sf cCart}~\subset \cA_0
~,\qquad
{\cA_-}^{\sf cCart}~\subset \cA_-
~,\qquad
{\cA_+}^{\sf cCart}~\subset \cA_+
~,\qquad
{\cA}^{\sf cCart}~\subset \cA
~.
\]

\end{notation}

\begin{remark}
Using Observation~\ref{t41}, note that none of the functors $\cA_0\ra \bcD(\SS^1)$ or $\cA_{\pm}\ra\bcD(\SS^1)$ are coCartesian fibrations.

\end{remark}

\begin{lemma}\label{t43}
Each of the diagrams among $(\infty,1)$-categories,
\[
\xymatrix{
\cA_0  \ar[rr] \ar[d]
&&
\cA_+ \ar[d]
&&
&&
{\cA_0}^{\sf cCart}  \ar[rr] \ar[d]
&&
{\cA_+}^{\sf cCart} \ar[d]
\\
\cA_- \ar[rr]
&&
\cA
&&
\text{ and }
&&
{\cA_0}^{\sf cCart}  \ar[rr] 
&&
{\cA_+}^{\sf cCart} 
,
}
\]
is a pushout.

\end{lemma}

\begin{proof}
We prove the left square is a pushout first.
Recall that the restricted Yoneda functor
\[
\Cat_1 \xra{\Hom_{\Cat_1}\bigl( [\bullet],-\bigr) }\PShv(\bDelta)
\]
is fully-faithful.
So the left square is a pushout in $\Cat_1$ if it is carried by this restricted Yoneda functor to a pushout square in $\PShv(\bDelta)$.
As colimits are detected and created value-wise in a presheaf $(\oo,1)$-category, it is therefore sufficient to show, for each object $[p]\in \bDelta$, that the square among spaces,
\[
\xymatrix{
\Hom_{\Cat_1}\bigl( [p] , \cA_0 \bigr)  \ar[rr] \ar[d]
&&
\Hom_{\Cat_1}\bigl( [p] , \cA_+ \bigr) 
\ar[d]
\\
\Hom_{\Cat_1}\bigl( [p] , \cA_- \bigr) 
\ar[rr]
&&
\Hom_{\Cat_1}\bigl( [p] , \cA \bigr) 
,
}
\]
is a pushout. 
By definition, each of the $(\oo,1)$-categories $\cA_{\pm}\subset \cA$ are full $\infty$-subcategories of $\cA$, and the $(\oo,1)$-category $\cA_0\simeq \cA_-\underset{\cA}\times \cA_+\subset \cA$ is their intersection in $\cA$, so is also a full $\infty$-subcategory of $\cA$.
Consequently, the above square of spaces is a pullback square among spaces in which each map is an inclusion of path-components.
Such a square of spaces is a pushout provided the induced map between sets of path-components
\begin{equation}\label{e1}
\pi_0\Bigl( 
\Hom_{\Cat_1}\bigl( [p] , \cA_- \bigr) 
\Bigr)
\coprod
\pi_0\Bigl( 
\Hom_{\Cat_1}\bigl( [p] , \cA_+ \bigr)
\Bigr)
\longrightarrow 
\pi_0\Bigl( 
\Hom_{\Cat_1}\bigl( [p] , \cA \bigr) 
\Bigr)
\end{equation}
is surjective. 
So let $[p]\xra{\sigma} \cA$ represent a point in the codomain of this map.
Inspecting the definitions of $\cA_{\pm}$, the induced map between sets of path-components
\[
\pi_0\Bigl( 
\Obj( \cA_- )
\Bigr)
\coprod
\pi_0\Bigl(
\Obj( \cA_+ )
\Bigr)
\longrightarrow 
\pi_0\Bigl( 
\Obj( \cA ) 
\Bigr)
\]
is surjective. 
So $\sigma(0)\in \cA$ belongs to either $\cA_-$ or $\cA_+$.
Observation~\ref{t40} implies $\sigma\in \Hom_{\Cat_1}\bigl( [p] , \cA\bigr)$ belongs to either $\Hom_{\Cat_1}\bigl( [p] , \cA_- \bigr)$ or $\Hom_{\Cat_1}\bigl( [p] , \cA_+ \bigr)$.
Therefore~(\ref{e1}) is surjective, which completes the proof that the left square in the lemma is a pushout.   

The proof that the right square in the lemma is a pushout follows the same logic.  
\end{proof}

Lemma~\ref{t43} has the following immediate consequence, toward our goal of computing $\int_{\SS^1}\Adj$.

\begin{cor}\label{cor.pushout.loc}
In the canonical diagram among $(\infty,1)$-categories,
\[
\xymatrix{
\cA_0\bigl[(\cA_0^{\sf cCart})^{-1}\bigr]
\ar[rr] 
\ar[d]
&&
\cA_+\bigl[(\cA_+^{\sf cCart})^{-1}\bigr]
\ar[d]
\\
\cA_-\bigl[(\cA_-^{\sf cCart})^{-1}\bigr]
\ar[rr]
&&
\cA\bigl[(\cA^{\sf cCart})^{-1}\bigr]
\ar[rr]^-{\simeq}_-{{\rm Obs}~\ref{t39}} 
&&
\displaystyle\int_{\SS^1} \Adj
~,
}
\]
the square is a pushout.
\end{cor}

The next result supplies means to identify the localizations in the pushout square of Corollary~\ref{cor.pushout.loc} above.

\begin{lemma}\label{t44}
Each of the fully-faithful inclusions over $\bcD(\SS^1)$ admits a left adjoint
\[
\xymatrix{
\cA_+
\ar@(u,-)@{-->}[rr]
&&
\displaystyle\int^{\lax}_{\SS^1} \fB \un{\End}_{\Adj}(+)
\ar[ll]^-{\rm inclusion}
&
\text{ and }
&
\cA_-
\ar@(u,-)@{-->}[rr]
&&
\displaystyle\int^{\lax}_{\SS^1} \fB \un{\End}_{\Adj}(-)
\ar[ll]^-{\rm inclusion}
~.
}
\]
These restrict as left adjoints
{\Small
\[
\xymatrix{
\cA_0
\ar@(u,-)@{-->}[rr]^-{}
&&
\displaystyle\int^{\lax}_{\SS^1} \fB \un{\End}_{\Adj}(+) \smallsetminus \bcD(\SS^1)
\ar[ll]^-{\rm inclusion}
&
\text{ and }
&
\cA_0
\ar@(u,-)@{-->}[rr]
&&
\displaystyle\int^{\lax}_{\SS^1} \fB \un{\End}_{\Adj}(-) \smallsetminus \bcD(\SS^1)
\ar[ll]^-{\rm inclusion}
~.
}
\]
}
Furthermore, the unit transformation for each of these adjunctions is by coCartesian morphisms over $\bcD(\SS^1)$.
Moreover, each of these left adjoints detects coCartesian morphisms over $\bcD(\SS^1)$.

\end{lemma}

\begin{proof}
We focus on the existence of, and properties of, the lefthand adjoints. The righthand situation follows by the same logic.
We first show that the upper left adjoint exists.
For this, we must show, for each object $(D,\ell)\in \cA_+$, that the $\infty$-undercategory $\Bigl( \int^{\lax}_{\SS^1} \fB \un{\End}_{\Adj}(+) \Bigr)^{(D,\ell)/}$ has an initial object.

So fix $(D,\ell)\in \cA_+$. 
Recall that $D \xra{\sf ref} \SS^1$ is a disk-refinement onto the circle. This refinement is uniquely determined by its resulting non-empty finite subset $\sk_0(D) \subset \SS^1$.  The functor $\fC(D) \xra{\ell} \Adj$ is uniquely determined by its resulting maps among sets 
{\Small
\[
\ell_0\colon 
\sk_0(D) = \Obj\bigl(\fC(D)\bigr) \xra{ \Obj(\ell) } \Obj(\Adj) = \{\pm\}
\qquad\text{ and }\qquad
\ell_1\colon 
\pi_0\bigl( D \smallsetminus \sk_0(D) \bigr)
\hookrightarrow 
\Mor\bigl( \fC(D) \bigr)
\xra{\Mor(\ell)}
\Mor(\Adj)
\]
}
that are source-target compatible -- so long as $+$ is one of its values.
Consider the morphism $(D,\ell) \to (D_+,\ell_+)$ in $\cA_+$ that is the coCartesian lift over the refinement morphism $D\to D_+$ in $\bcD(\SS^1)$ that is characterized by the induced diagram among spaces of objects,
\[
\xymatrix{
\sk_0(D_+) \ar[rr]^-{\ell_+} \ar[d]
&&
\{+\} \ar[d]^-{\rm inclusion}
\\
\sk_0(D) \ar[rr]^-{\ell_0}
&&
\{\pm\}
,
}
\]
being a pullback.  
(The assumption that $+$ is a value of $\ell_0$ ensures that $\sk_0(D_+)\neq\emptyset$ is not empty, so $D_+$ is indeed a disk-refinement of $\SS^1$.)
Because the resulting map $(\ell_+)_0\colon \sk_0(D_+) \to \{\pm\}$ is constant at $+$, Observation~\ref{t41} grants that this object $(D_+,\ell_+)$ belongs to the image of the fully-faithful functor $\int^{\lax}_{\SS^1} \fB \un{\End}_{\Adj}(+) \hookrightarrow \cA_+$ which is the putative right adjoint.  
So $(D_+,\ell_+)$ canonically defines an object in the undercategory 
$\Bigl( \int^{\lax}_{\SS^1} \fB \un{\End}_{\Adj}(+) \Bigr)^{(D,\ell)/}$.  
We will show this object $(D_+,\ell_+)$ is initial.

So let $(D,\ell) \xra{(X,\gamma)}(D',\ell')$ be an object in $\Bigl( \int^{\lax}_{\SS^1} \fB \un{\End}_{\Adj}(+) \Bigr)^{(D,\ell)/}$.  
We will show there is a unique morphism $(D_+,\ell_+) \to (D',\ell')$ in this undercategory.  
Recall the identification $\sk_0\colon  \bcD(\SS^1)\xra{\simeq} \para^{\op}$ of Corollary~\ref{t21}.  
The ``surjective followed by injective'' factorization system on $\para$ determines a ``refinement followed by creation'' factorization system on $\bcD(\SS^1)$.  
We can thusly reduce to the case in which the morphism $(D,\ell) \xra{(X,\gamma)} (D',\ell')$ lies over a refinement morphism in $\bcD(\SS^1)$.
By assumption, the map $\ell'_0\colon \sk_0(D') \to \{\pm\}$ is constant at $+$.
So there is a commutative diagram among sets
\[
\xymatrix{
\sk_0(D') \ar[rr]^-{\ell'_0} \ar[d]_-{\sk_0(X)}
&&
\{+\} \ar[d]^-{\rm inclusion}
\\
\sk_0(D) \ar[rr]^-{\ell_0}
&&
\{\pm\}
}
\]
in which the left vertical map is an inclusion induced by the given refinement $D\to D'$.  
By construction of $D_+$, there is a unique factorization $D \to D_+ \to D'$ of the given refinement morphism $D\to D'$ in $\bcD(\SS^1)$.
Because the morphism $(D,\ell)\to (D_+,\ell_+)$ is defined to be coCartesian over $\bcD(\SS^1)$, then there exists a unique factorization $(D,\ell) \to (D_+,\ell_+) \to (D',\ell')$ of the given morphism $(D,\ell)\xra{(X,\gamma)} (D',\ell')$ in $\cA_+$.  
This completes the verification of the existence of the sought left adjoint in the upper left display of the lemma.
Notice that the unit of this adjunction, $(D,\ell)\to (D_+,\ell_+)$, is by coCartesian morphisms over $\bcD(\SS^1)$.

We now show that the left adjoint in the upper left of the display of the lemma restricts as a left adjoint in the bottom left of the display.  
For this note that an object $(D,\ell)\in \cA$ belongs to $\cA_0$ if and only if the corresponding map between finite sets $\ell_0\colon \sk_0(D) \to \{\pm\}$ is surjective.  
In particular, the functor $\fC(D) \xra{\ell} \Adj$ is not constant.  
Now, inspecting the value $(D_+,\ell_+) \in \int^{\lax}_{\SS^1}\fB \un{\End}_{\Adj}(+)$ of the left adjoint just constructed, should $(D,\ell)$ belong to $\cA_0$ then the functor $\ell_+ \colon \fC(D_+) \hookrightarrow \fC(D) \xra{\ell} \Adj$ is not constant.  
Therefore, this object $(D_+,\ell_+)\in \int^{\lax}_{\SS^1} \fB \un{\End}_{\Adj}(+)$ does not belong to $\bcD(\SS^1) \xra{\sf unit} \int^{\lax}_{\SS^1} \fB \un{\End}_{\Adj}(+)$.

It remains to prove the final statement of the lemma.
So let $(D,\ell) \xra{(X,\gamma)} (D',\ell')$ be a morphism in $\cA_+$.
The unit transformation for the adjunction constructed just above places this morphism in a commutative square in $\cA_+$:
\[
\xymatrix{
(D,\ell) \ar[rr]^-{(X,\gamma)}  \ar[d]_-{\rm unit}
&&
(D',\ell') \ar[d]^-{\rm unit}
\\
(D_+,\ell_+) \ar[rr]^-{(X_+,\gamma_+)}
&&
(D'_+,\ell'_+)
.
}
\]
We must show that the morphism $(X,\gamma)$ is coCartesian over $\bcD(\SS^1)$ if and only if the morphism $(X_+,\gamma_+)$ is coCartesian over $\bcD(\SS^1)$.
As noted above, the vertical unit morphisms are coCartesian over $\bcD(\SS^1)$.  
From a 2-of-3 property for coCartesian morphisms, if the morphism $(X,\gamma)$ is coCartesian over $\bcD(\SS^1)$, then so is the morphism $(X_+,\gamma_+)$.  
Now suppose the morphism $(X,\gamma)$ is \emph{not} coCartesian over $\bcD(\SS^1)$.
As in Remark~\ref{r7}, this is to say that the 2-cell $\gamma$ is \emph{not} invertible.  
Using that $\cA \to \bcD(\SS^1)$ is a coCartesian fibration, there is a unique factorization $(D,\ell) \to (D',\ell'') \to (D',\ell')$ in $\cA$ of the morphism $(X,\gamma)$ by a coCartesian morphism followed by a morphism over an equivalence in $\bcD(\SS^1)$.  
Because $\cA_+ \to \cA$ is a fully-faithful right fibration, this factorization lies in $\cA_+$.
So we can reduce to the case in which the given morphism $(D,\ell) \xra{(X,\gamma)} (D',\ell')$ lies over an equivalence $D\xra{\simeq} D'$ in $\bcD(\SS^1)$.  
Finally, the explicit description of $\Adj$ of Theorem~\ref{t7} reveals that the composition functor $\un{\Hom}_{\Adj}(a,b) \times \un{\Hom}_{\Adj}(b,c) \xra{\circ} \un{\Hom}_{\Adj}(a,c)$ is conservative.  
It follows that, if the 2-cell $\gamma$ of the morphism $(D,\ell) \xra{(X,\gamma)} (D',\ell')$ is not invertible, then neither is the 2-cell of the morphism $(D_+,\ell_+) \xra{(X_+,\gamma_+)} (D'_+,\ell'_+)$.  
As in Remark~\ref{r7}, we conclude that the morphism $(X_+,\gamma_+)$ is \emph{not} coCartesian over $\bcD(\SS^1)$.
\end{proof}

\begin{cor}\label{cor.last}
The adjunctions of Lemma~\ref{t44} determine canonical equivalences
\[
\xymatrix{
\cA_+\bigl[(\cA_+)^{-1}\bigr]
~\simeq~
\displaystyle\int_{\SS^1}\fB\un{\End}_{\Adj}(+)
&
{\text and}
&
\cA_-\bigl[(\cA_-)^{-1}\bigr]
~\simeq~
\displaystyle\int_{\SS^1}\fB\un{\End}_{\Adj}(-)}
\]
which restrict as equivalences
\[
\Bigl(\displaystyle\int_{\SS^1}\fB\un{\End}_{\Adj}(+)\Bigr)\smallsetminus\ast
~\simeq~
\cA_0\bigl[(\cA_0)^{-1}\bigr]
~\simeq~
\Bigl(\displaystyle\int_{\SS^1}\fB\un{\End}_{\Adj}(-)\Bigr)\smallsetminus\ast
\]
\end{cor}
\begin{proof}
We first establish the top left canonical equivalence; the top right canonical equivalence follows the same logic.
Lemma~\ref{t44} implies the adjunctions of that lemma extend as an adjunction in the $(\infty,2)$-category $\Ar(\Cat_1)$:
\begin{equation}\label{e2}
\xymatrix{
\Bigl(\cA_+^{\sf cCart}\ra\cA_+\Bigr)
\ar@(u,-)[rr]
&&
\Bigl(\Bigl(\displaystyle\int^{\lax}_{\SS^1} \fB \un{\End}_{\Adj}(+)\Bigr)^{\sf cCart}
\ra
\displaystyle\int^{\lax}_{\SS^1} \fB \un{\End}_{\Adj}(+)
\ar[ll]^-{\rm inclusion}}\Bigr)
~.
\end{equation}
The functor between $(\infty,2)$-categories, $\Ar(\Cat_1) \xra{(\cW \to \cC)\mapsto \sB \cW \underset{\cW} \amalg \cC} \Cat_1$, carries the above adjunction to an adjunction:
{\small
\begin{equation}\label{e3}
\xymatrix{
\cA_+\Bigl[(\cA_+^{\sf cCart})^{-1}\Bigr]
\ar@(u,-)[rr]
&&
\displaystyle\int^{\lax}_{\SS^1} \fB \un{\End}_{\Adj}(+)\Bigl[\Bigl(\Bigl(\displaystyle\int^{\lax}_{\SS^1} \fB \un{\End}_{\Adj}(+)\Bigr)^{\sf cCart}\Bigr)^{-1}\Bigr]
\underset{\rm Obs~\ref{t39}}{~\simeq~}
\int_{\SS^1}\fB\un{\End}_{\Adj}(+)~.
\ar[ll]}
\end{equation}
}
Lemma~\ref{t44} implies unit and counit 2-morphisms in the adjunction~(\ref{e2}) are by coCartesian morphisms over $\bcD(\SS^1)$.
It follows that the unit and counit 2-morphisms in the adjunction~(\ref{e3}) are by equivalences.
So the adjunction~(\ref{e3}) is an equivalence between $(\oo,1)$-categories.

We now show that the top left canonical equivalence restricts as the bottom left canonical equivalence; the bottom right canonical equivalence follows the same logic.
By the same logic above, there is a canonical equivalence between $(\oo,1)$-categories:
{\small
\[
\cA_0\Bigl[(\cA_0^{\sf cCart})^{-1}\Bigr]
~\simeq~
\Bigl(\displaystyle\int^{\lax}_{\SS^1} \fB \un{\End}_{\Adj}(+)\smallsetminus\bcD(\SS^1)\Bigr)\Bigl[\Bigl(\Bigl(\displaystyle\int^{\lax}_{\SS^1} \fB \un{\End}_{\Adj}(+)\Bigr)^{\sf cCart}\smallsetminus\bcD(\SS^1)\Bigr)^{-1}\Bigr]
~,
\]
}
where $\bcD(\SS^1)$ is identified as in~(\ref{e61}) with the $\oo$-subcategory of $\int^{\lax}_{\SS^1}\fB\un{\End}_{\Adj}(+)$ in which every 1-stratum is labeled by the identity morphisms of the object $+\in \Adj$ (see Observation~\ref{t41}).
Observation~\ref{t40} implies there are no morphisms from objects in the complement $\int^{\lax}_{\SS^1} \fB \un{\End}_{\Adj}(+)\smallsetminus\bcD(\SS^1)$ to objects in $\bcD(\SS^1)$. 
It follows that the localization of the complement is the complement of the localizations, which establishes the first equivalence between $(\oo,1)$-categories:
{\small
\begin{eqnarray}
\nonumber
\Bigl(\displaystyle\int^{\lax}_{\SS^1} \fB \un{\End}_{\Adj}(+)\smallsetminus\bcD(\SS^1)\Bigr)\Bigl[\Bigl(\Bigl(\displaystyle\int^{\lax}_{\SS^1} \fB \un{\End}_{\Adj}(+)\Bigr)^{\sf cCart}\smallsetminus\bcD(\SS^1)\Bigr)^{-1}\Bigr]
&
\simeq
&
\\
\nonumber
\displaystyle\int^{\lax}_{\SS^1} \fB \un{\End}_{\Adj}(+)\Bigl[\Bigl(\Bigl(\displaystyle\int^{\lax}_{\SS^1} \fB \un{\End}_{\Adj}(+)\Bigr)^{\sf cCart}\Bigr)^{-1}\Bigr]
\smallsetminus 
\bcD(\SS^1)\Bigl[\bcD(\SS^1)^{-1}\Bigr]
&
\xra{\simeq}
&
\int_{\SS^1}\fB\un{\End}_{\Adj}(+)\smallsetminus \ast
~.
\end{eqnarray}
}
The second equivalence is the canonical functor.
It is an equivalence by Observation~\ref{t39}, together with contractibility of the $\infty$-groupoid-completion $\bcD(\SS^1)\Bigl[\bcD(\SS^1)^{-1}\Bigr] \simeq \ast$  (see Corollary~\ref{t22}).
\end{proof}

\begin{proof}[Proof of Theorem~\ref{t10}]
We explain the following diagram among $(\infty,1)$-categories, which consolidates the preceding results of this section:
{\Small
\[
\xymatrix{
\para \ar[rr]^-{\rm inclusion}
&&
\para^{\tl} 
&&
\\
\Bigl(\int_{\SS^1} \fB \un{\End}_{\Adj}(-)\Bigr) \smallsetminus \ast
\ar[rr]^-{\rm inclusion}
\ar[dd]^-{\simeq}_{\rm Cor~\ref{cor.last}}
\ar[u]^-{\rm Cor~\ref{t30}}_-{\simeq}
&&
\int_{\SS^1} \fB \un{\End}_{\Adj}(-)
\ar[d]^-{\simeq}_{\rm Cor~\ref{cor.last}}
\ar[u]^-{\rm Cor~\ref{t30}}_-{\simeq}
&&
\\
&&
\cA_+\bigl[ ({\cA_+}^{\sf cCart})^{-1} \bigr]  
\ar[drr]
&&
\\
\cA_0\bigl[ ({\cA_0}^{\sf cCart})^{-1} \bigr]  
\ar[drr] \ar[urr]
&&
&&
\cA\bigl[(\cA^{\sf cCart})^{-1}\bigr]\simeq \int_{\SS^1} \Adj
\\
&&
\cA_-\bigl[ ({\cA_-}^{\sf cCart})^{-1} \bigr]  
\ar[urr]
&&
\\
\Bigl(\int_{\SS^1} \fB \un{\End}_{\Adj}(+)\Bigr) \smallsetminus \ast
\ar[rr]^-{\rm inclusion}
\ar[uu]_-{\simeq}^{\rm Cor~\ref{cor.last}}
\ar[d]_-{\rm Cor~\ref{t30}}^-{\simeq}
&&
\int_{\SS^1} \fB \un{\End}_{\Adj}(+)
\ar[u]_-{\simeq}^{\rm Cor~\ref{cor.last}}
\ar[d]_-{\rm Cor~\ref{t30}}^-{\simeq}
&&
\\
\para^{\op} \ar[rr]^-{\rm inclusion}
&&
(\para^{\op})^{\tr}
&&
.
}
\]
}
Corollary~\ref{cor.pushout.loc} states that the inner diamond is a pushout.
Corollary~\ref{t30} supplies the equivalences labeled as so, as well as commutativity of the squares in which they participate.  
Corollary~\ref{cor.last} supplies the equivalences labeled a so, as well as commutativity of the squares in which they participate. 
Through the left vertical equivalence $\para \simeq \para^{\op}$ given by combinatorial Poincar\'e duality (see Remark~\ref{rem.PD}), 
we conclude a canonical identification
\[
\para^{\tl} \underset{\para} \coprod \para^{\tr} 
\xra{~\simeq~}
\int_{\SS^1} \Adj
~.
\]
Finally, Lemma~\ref{t53} supplies an equivalence 
$\para^{\tl} \underset{\para} \coprod \para^{\tr} 
\xra{\simeq}
\para^{\tl\!\tr}$, 
which completes the proof.  
\\

\end{proof}

\section{Traces}\label{sec.traces}
Here we use Theorem~\ref{t10}, which identifies the factorization homology over the circle of the walking adjunction, to homotopy-coherently construct a $\TT$-invariant trace map from the Hochschild homology of an endomorphism algebras.  
This is articulated in Theorem~\ref{r9}, whose proof occupies~\S\ref{sec.traces.1}.
In \S\ref{sec.traces.2}, we use Theorem~\ref{r9} to prove a conjecture of To\"en--Vezzosi.

\subsection{$\TT$-invariant trace}\label{sec.traces.1}

\begin{convention}\label{d2}
In this subsection, $\bfX$ is a symmetric monoidal $(\infty,1)$-category.
\end{convention}

This section assembles results of the previous section to construct \bit{trace} morphisms. 
Using factorization homology, we obtain, for each dualizable object $V\in \bfX$, a $\sB\TT$-family of endomorphisms of its unit:
\[
\Bigl(
\TT
\racts
\sZ_V(\SS^1)
\Bigr)
~\in~
\End_{\bfX}(\uno)^{\pitchfork \sB \TT}
\qquad
\bigl( {\rm see~Remark}~\ref{r10}\bigr)
~.
\]
Such a value is expected from the 1-dimensional oriented cobordism hypothesis, which predicts a 1-dimensional oriented TQFT determined by $V$.  
Constructing the above value, with its $\TT$-equivariance, is the most subtle among all values of such a TQFT.
This section, which makes no reference to the symmetric monoidal $(\infty,1)$-category $\Bord$, and thusly it does not logically depend on the 1-dimensional cobordism hypothesis.

We begin by observing how, using factorization homology, an associative (co)algebra naturally determines a diagram in $\bfX$.
Using Proposition~5.1 of~\cite{oldfact} applied to $\bfX\in \CAlg(\Cat_1)$, the unique map between spaces $\SS^1 \xra{!} \ast$ determines a functor
\begin{equation}\label{e6}
\underset{\bfX} \bigotimes
\colon
\int^\alpha_{\SS^1} \bfX
\xra{~\int^\alpha_{!} \bfX~}
\int^\alpha_\ast  \bfX
~\simeq~
\bfX
~.
\end{equation}

\begin{observation}\label{t2}
\begin{enumerate}
\item[~]

\item
Let $A \in \Alg(\bfX)$ be an associative algebra in $\bfX$.
Through the identification $\Alg ( \bfX ) \simeq \Fun^{\cE_1} ( \sO , \bfX )$ of Proposition~\ref{t33}, $A$ is a monoidal functor $\sO \xra{\lag A\rag} \bfX$.  
Applying factorization homology over $\SS^1$, and concatenating with Lemma~\ref{t31}(2) and Lemma~\ref{t11}, results in a diagram in $\bfX$:
\begin{equation}\label{e65}
\para^{\tl}
\underset{\rm Lem~\ref{t31}(2)}{~\simeq~}
\Disk^{\fr}_{1/\SS^1}
\underset{\rm Lem~\ref{t11}}{~\simeq~}
\int^\alpha_{\SS^1} \sO
\xra{~\int^\alpha_{\SS^1} \lag A\rag~}
\int^\alpha_{\SS^1} \bfX
\underset{(\ref{e6})}{\xra{~\underset{\bfX}\bigotimes~}}\bfX
~.
\end{equation}

\item
Let $C\in {\sf cAlg}(\bfX)$ be an associative coalgebra in $\bfX$.  
Through the identification ${\sf cAlg}( \bfX ) \simeq \Fun^{\cE_1} ( \sO^{\op} , \bfX )$ of Proposition~\ref{t33}, $C$ is a monoidal functor $\sO^{\op} \xra{\lag C \rag} \bfX$.  
Applying factorization homology over $\SS^1$, and concatenating with Lemma~\ref{t31}(2) and Lemma~\ref{t11}, results in a diagram in $\bfX$:
\begin{equation}\label{e67}
(\para^{\op})^{\tr}
\underset{\rm Lem~\ref{t31}(2)}{~\simeq~}
\bigl(\Disk^{\fr}_{1/\SS^1}\bigr)^{\op}
\underset{\rm Lem~\ref{t11}}{~\simeq~}
\int^\alpha_{\SS^1} \sO^{\op}
\xra{~\int^\alpha_{\SS^1} \lag C\rag~}
\int^\alpha_{\SS^1} \bfX
\underset{(\ref{e6})}{\xra{~\underset{\bfX}\bigotimes~}}
\bfX
~.
\end{equation}

\end{enumerate}
\end{observation}

\begin{observation}\label{t49}
Consider the set-up of Observation~\ref{t2}.
\begin{enumerate}
\item
This diagram~(\ref{e65}) in $\bfX$ is equivalent to that computing the $\alpha$-version of factorization homology over $\SS^1$, in the sense of Definition~3.2 of~\cite{oldfact}: 
\begin{equation}\label{e66}
\int^\alpha_{\SS^1} A
~\simeq~
\colim\Bigl(
\para^{\tl}
\underset{\rm Lem~\ref{t31}(2)}{~\simeq~}
\Disk^{\fr}_{1/\SS^1}
\underset{\rm Lem~\ref{t11}}{~\simeq~}
\int^\alpha_{\SS^1} \sO
\xra{~\int^\alpha_{\SS^1} A~}
\int^\alpha_{\SS^1} \bfX
\underset{(\ref{e6})}{\xra{~\underset{\bfX}\bigotimes~}}
\bfX
\Bigr)
~.
\end{equation}

\item
This diagram~(\ref{e67}) in $\bfX$ is equivalent to that computing the $\alpha$-version of factorization cohomology over $\SS^1$, in the sense of Definition~2.2.1 of~\cite{zp}: 
\begin{equation}\label{e68}
\int^{\SS^1}_\alpha C
~\simeq~
\limit\Bigl(
(\para^{\op})^{\tr}
\underset{\rm Lem~\ref{t31}(2)}{~\simeq~}
\bigl(\Disk^{\fr}_{1/\SS^1}\bigr)^{\op}
\underset{\rm Lem~\ref{t11}}{~\simeq~}
\int^\alpha_{\SS^1} \sO^{\op}
\xra{~\int^\alpha_{\SS^1} C~}
\int^\alpha_{\SS^1} \bfX
\underset{(\ref{e6})}{\xra{~\underset{\bfX}\bigotimes~}}
\bfX
\Bigr)
~.
\end{equation}

\end{enumerate}

\end{observation}

\begin{prop}\label{r8}
Consider the set-up of Observation~\ref{t2}.
\begin{enumerate}

\item
Provided the $(\oo,1)$-category $\bfX$ admits geometric realizations,
the Hochschild homology ${\sf HH}(A) := \int^\alpha_{\SS^1} A\in \bfX$ exists.
Furthermore, the unit of $A$ canonically determines a $\TT$-invariant morphism in $\bfX$:
\[
\uno
\xra{~{\sf unit}_A~}
{\sf HH}(A)
~.
\]

\item
Provided the $(\oo,1)$-category $\bfX$ admits totalizations,
the coHochschild homology ${\sf cHH}(C) := \int_\alpha^{\SS^1} C \in \bfX$ exists.
Furthermore, the counit of $C$ canonically determines a $\TT$-invariant morphism in $\bfX$:
\[
{\sf cHH}(C)
\xra{~{\sf counit}_C~}
\uno
~.
\]

\end{enumerate}

\end{prop}

\begin{proof}
The two statements imply one another, by replacing $\bfX$ with $\bfX^{\op}$.
So it is sufficient to prove the first statement.  
Observation~\ref{t3} supplies a final functor $\bDelta^{\op} \to \para^{\tl}$.
Therefore, the assumption that $\bfX$ admits geometric realizations ensures the colimit of~(\ref{e65}) exists.  
Observation~\ref{t49} implies this colimit of~(\ref{e65}) is the Hochschild homology of $A$.
Inspecting the construction of the functor~(\ref{e65}) reveals that the cone-point is carried to the symmetric monoidal unit in $\bfX$, and canonical morphism from this value to the colimit is implemented by the unit of $A$:
\[
\uno
\xra{~{\sf unit}_A~}
\colim\bigl( \para^{\tl} \to \bfX \bigr)
\underset{\rm Obs~\ref{t49}}{~\simeq~}
\displaystyle\int^\alpha_{\SS^1} \bfX
~=:~
{\sf HH}(A)
~.
\]
Because the cone-point in $\para^{\tl}$ is $\sB\ZZ$-invariant, this morphism is $\TT\simeq \sB\ZZ$-invariant.

\end{proof}

Now, let $V\in \Obj(\bfX^{\sf duals})$ be an object that is dualizable.  
Denote its dual as $V^\vee\in \Obj(\bfX^{\sf duals})$.  
With respect to the right-lax monoidal functor $\Hom_{\bfX}(\uno , -) \colon \bfX \to \Spaces$, the
the monoid of endomorphisms $\End_{\bfX}(V)$ canonically lifts as an associative algebra in $\bfX$;
with respect to the right-lax monoidal functor $\Hom_{\bfX}(- , \uno) \colon \bfX^{\op} \to \Spaces$, the
the monoid of endomorphisms $\End_{\bfX}(V)$ canonically lifts as an associative coalgebra in $\bfX$:
\begin{equation}\label{e4}
\un{\End}(V):=V^\vee \ot  V
~\in~
\Alg(\bfX)
\qquad\text{ and }\qquad
\un{\End}(V):= V\ot V^\vee
~\in~
{\sf cAlg}(\bfX)
~.
\end{equation}
This associative algebra structure and associative coalgebra structure are intertwined: they are the associated monad and comonad of an adjunction\footnote{This implies that $\un{\End}(V)$ has the structure of a Frobenius algebra.}
in $\fB \bfX$, as we now explain.

Recall the identification $\Hom_{\fCat_2}\bigl( \Adj , \fB \bfX \bigr) \simeq \Obj(\bfX^{\sf duals})$ of Observation~\ref{t15}.
Through this identification, $V$ is the datum of a functor between category-objects internal to $(\infty,1)$-categories: $\Adj \xra{\lag V\rag} \fB \bfX$.  
The resulting composite functors
\[
\lag \un{\End}(V) \rag 
\colon 
\fB \sO \underset{\rm Thm~\ref{t7}}{~\hookrightarrow~}
\Adj
\xra{~\lag V \rag~}
\fB \bfX
\qquad\text{ and }\qquad
\lag \un{\End}(V) \rag 
\colon
\fB \sO^{\op} \underset{\rm Thm~\ref{t7}}{~\hookrightarrow~}
\Adj
\xra{~\lag V \rag~}
\fB \bfX
\]
select the associative algebra $\un{\End}(V)$ and the associative coalgebra $\un{\End}(V)$, respectively.

\begin{observation}\label{t50}
The diagrams~(\ref{e65}) and~(\ref{e67}) in $\bfX$ extend to a $\int_{\SS^1}\Adj$-indexed diagram in $\bfX$:
\begin{equation}\label{e69}
\xymatrix{
(\ref{e65})\colon
\int_{\SS^1} \fB \sO
\ar[rr]^-{\rm Prop~\ref{t23}}_-{\simeq}
\ar[dr]_-{\int_{\SS^1} {\rm Thm~\ref{t7}}}
&&
\int^\alpha_{\SS^1} \sO
\ar[drr]^-{{}~{}~{}~{}~{}~\int^\alpha_{\SS^1} \lag \underline{\End}_{\bfX}(V) \rag }
&
&&
\\
&
\int_{\SS^1} \Adj
\ar[r]^-{\int_{\SS^1} \lag V \rag }
&
\int_{\SS^1} \fB \bfX
\ar[rr]^-{\rm Prop~\ref{t23}}_-{\simeq}
&&
\int^\alpha_{\SS^1\smallsetminus \sk_0(\SS^1)} \bfX
\ar[r]^-{~\underset{\bfX}\bigotimes~}_-{(\ref{e6})}
&
\bfX
\\
(\ref{e67})\colon
\int_{\SS^1} \fB \sO^{\op}
\ar[rr]^-{\rm Prop~\ref{t23}}_-{\simeq}
\ar[ur]^-{\int_{\SS^1} {\rm Thm~\ref{t7}}}
&&
\int^\alpha_{\SS^1} \sO^{\op}
\ar[urr]_-{{}~{}~{}~{}~{}~\int^\alpha_{\SS^1} \lag \underline{\End}_{\bfX}(V) \rag }
&
&&
.
}
\end{equation}
\end{observation}

\begin{theorem}\label{r9}
Let $V\in \bfX$ be a dualizable object in a symmetric monoidal $(\oo,1)$-category.
\begin{enumerate}
\item
If $\bfX$ admits geometric realizations,
there exists a $\TT$-invariant \bit{trace} morphism
from Hochschild homology of the associative algebra $\un{\End}(V)$ in $\bfX$:
\[
{\sf HH}\bigl( \un{\End}(V) \bigr)
\xra{~\sf trace~}
\uno
~.
\] 

\item
If $\bfX$ admits totalizations,
there exists a $\TT$-invariant \bit{cotrace} morphism from the coHochschild homology of the associative coalgebra $\un{\End}(V)$ in $\bfX$:
\[
\uno
\xra{~\sf cotrace~}
{\sf cHH}\bigl( \un{\End}(V) \bigr)
~.
\] 

\item
If $\bfX$ admits both geometric realizations and totalizations,
then the two resulting composite $\TT$-invariant endomorphisms of $\uno\in \bfX$ are identical.
Namely, the diagram in $\bfX$ canonically commutes:
{\small
\begin{equation}\label{e5}
\xymatrix{
&&
{\sf HH}\bigl( \un{\End}(V) \bigr)
\ar[drr]^{ {\sf trace} }
&&
&
\\
\uno
\ar[urr]^{{\sf unit}_{\un{\End}(V)} }
\ar[drr]_-{ {\sf cotrace} }
&&
&&
\uno
&
\\
&&
{\sf cHH}\bigl( \un{\End}(V) \bigr)
\ar[urr]_{{\sf counit}_{\un{\End}(V)}}
&&
&
~.
}
\end{equation}
}
Furthermore, forgetting $\TT$-invariance, this common composite endomorphism of $\uno \in \bfX$ is the composite
\[
\uno
\xra{~\eta~}
V^\vee \ot V
~\simeq~
V\ot V^\vee
\xra{~\epsilon~}
\uno
~,
\]
where $\eta$ and $\epsilon$ are the respective unit and counit of the duality between $V$ and $V^\vee$.  

\end{enumerate}

\end{theorem}

\begin{proof}
Theorem~\ref{t10} identifies the outermost terms in the diagram~(\ref{e69}) among $(\oo,1)$-categories as
\begin{equation}\label{e70}
\xymatrix{
\para^{\tl}
\ar[dr]
\ar@/^1.2pc/[drrrrrr]
&&&&&
\\
&
\para^{\tl\!\tr}
\ar[rrrrr]
&&&&&
\bfX
\\
(\para^{\op})^{\tr}
\ar@/_1pc/[urrrrrr]
\ar[ur]^-{\rm Obs~\ref{t34}}
&&
&&
&&
.
}
\end{equation}
In this diagram~(\ref{e70}), each cone-point is carried to the symmetric monoidal unit $\uno \in \bfX$.
From the universal property of colimits, the upper triangle in this diagram~(\ref{e70}) determines a morphism in $\bfX$:
\begin{equation}\label{e8}
{\sf HH}\bigl( \un{\End}(V) \bigr)
\underset{\rm Obs~\ref{t49}}{~\simeq~}
\colim\bigl( \para^{\tl} \to \bfX)
\longrightarrow
\uno
~,
\end{equation}
provided this colimit exists.
Observation~\ref{t3} ensures that the colimit of the top functor to $\bfX$ exists provided $\bfX$ admits geometric realizations.
Because the cone-point in $\para^{\tr}$ is $\sB\ZZ$-invariant, this morphism~(\ref{e8}) is $\TT\simeq \sB\ZZ$-invariant.  
This proves Statement~(1).
Statement~(2) is implied by Statement~(1), by replacing $\bfX$ by $\bfX^{\op}$.

It remains to prove Statement~(3).
Assume $\bfX$ admits geometric realizations and totalizations.
Then Statements~(1) and~(2), together with Proposition~\ref{r8}, supply the diamond~(\ref{e5}) in $\bfX$.
By construction of this diamond, both of the composite morphisms $\uno \xra{\trace~ \circ~ \unit}\uno$ and $\uno \xra{{\sf cotrace} ~\circ~ {\sf counit}} \uno$ are the value on the unique morphism between cone-points in $\para^{\tl\!\tr}$ of the horizontal functor in~(\ref{e70}).
In particular, these two $\TT$-invariant endomorphisms of $\uno$ are canonically identified.

Finally, observe that the unique morphism $c_1\xra{\lag -\infty \xra{!}+\infty \rag} \para^{\tl\!\tr}$ uniquely factors through any object $\lambda \in \para$.  
In particular, it uniquely factors through the object $(\ZZ \racts \ZZ)\in \para$, which corresponds through Corollary~\ref{t21} to the disk-refinement of $\SS^1$ with a single connected 1-dimensional stratum.
This object in $\para$ is carried to $\un{\End}_{\bfX}(V) = V\ot V^\vee \in \bfX$.
Furthermore, the unique factorization $-\infty \xra{!}(\ZZ \racts \ZZ) \xra{!} +\infty$ in $\para^{\tl\!\tr}$ is carried to 
\[
\uno
\xra{~\eta~}
V^\vee \ot V
\simeq
V\ot V^\vee
\xra{~\epsilon~}
\uno
~.
\]

\end{proof}

\begin{remark}\label{r10}

The common $\TT$-invariant of endomorphisms of $\uno$ of Theorem~\ref{r9} is the value of a 1-dimensional TQFT $\sZ_V$ on the circle:
\[
\End_{\Bord}(\emptyset)^{\pitchfork  {\sf BDiff}^{\fr}(\SS^1)}
\ni
\bigl(
\emptyset
\xra{~\SS^1~}
\emptyset
\bigr)
~{}~{}~
\overset{\sZ_V}{~\mapsto~}
~{}~{}~
\bigl(
\uno
\xra{~\sZ_V(\SS^1)~}
\uno
\bigr)
~
\in 
\End_{\bfX}(\uno)^{\pitchfork \sB \TT}
~.
\]
Supplying this TQFT, with this value, is the subject of \cite{tang1}.
\\

\end{remark}

\subsection{A proof of a conjecture of To\"en--Vezzosi}\label{sec.traces.2}
We now prove Conjecture~5.1 of~\cite{toen.vezzosi} (which is shown there to imply Conjecture~4.1 of the same article).
We first recall some notation.
\begin{notation}
\begin{enumerate}

\item
[~]

\item
The full $\infty$-subcategory $\CAlg(\Cat_1)^{\sf rigid}\subset \CAlg(\Cat_1)$ consists of the \bit{rigid} symmetric monoidal $(\infty,1)$-categories (i,e,. those symmetric monoidal $(\infty,1)$-categories in which each object is dualizable).

\item
The $(\oo,1)$-category $\Spaces^{\pitchfork \sB \TT} = \Fun(\sB \TT, \Spaces)$ is that of spaces with a $\TT$-action.

\item
The \bit{moduli space of objects} functor is
\[
\Obj \colon \CAlg(\Cat_1)^{\sf rigid}
\longrightarrow
\Spaces
~,\qquad
\bfX
\mapsto 
\Obj(\bfX)
~.
\]

\item
The \bit{free loop} functor is
\[
L \colon \CAlg(\Cat_1)^{\sf rigid}
\longrightarrow
\Spaces^{\pitchfork \sB \TT}
~,\qquad
\bfX
\mapsto 
\Obj(\bfX)^{\pitchfork \SS^1} = \Map\bigl( \SS^1 , \Obj(\bfX) \bigr)
~,
\]
where $\Obj(\bfX)^{\pitchfork \SS^1}$ is the moduli space of objects, in $\bfX$, together with an automorphism.
Precomposing by the unique map $\SS^1\xra{!} \ast$ determines a natural transformation
\[
\Obj 
\xra{~\rm constant~}
L
\]
in which the domain is regarded as taking in $\Spaces^{\pitchfork \sB \TT}$ via the functor $\Spaces \xra{\rm trivial} \Spaces^{\pitchfork \sB \TT}$.

\item
The \bit{categorical based loop} functor is the composition
\[
\End(\uno) \colon \CAlg(\Cat_1)^{\sf rigid}
\xra{~\End(\uno)~}
\Spaces
\xra{~\rm trivial~}
\Spaces^{\pitchfork \sB \TT}
~,\qquad
\bfX
\mapsto 
\End_{\bfX}(\uno)
~.
\]
So each $\TT$-space $\End_{\bfX}(\uno)$ is endowed with the trivial $\TT$-action.

\end{enumerate}

\end{notation}

\begin{cor}[Conjecture~5.1 of \cite{toen.vezzosi}]\label{cor.toenvezzosi}
There is a natural transformation between functors from $\CAlg(\Cat_1)^{\sf rigid}$ to $\Spaces^{\pitchfork \sB \TT}$,
\[
L \xra{~\trace~} \End(\uno)
~,
\]
with the property that the composite natural transformation evaluates as
\[
\Obj
\xra{~\rm constant~}
L
\xra{~\trace~}
\End(\uno)
~,
\qquad
\bigl(V\in \Obj(\bfX)\bigr)
~\mapsto~ 
\Bigl(\bigl(\uno \xra{\eta} V^\vee \ot V \simeq V\ot V^{\vee} \xra{\epsilon} \uno \bigr)\in \End_{\bfX}(\uno) \Bigr)
~.
\]

\end{cor}

\begin{proof}
Because $\bfX$ is rigid, then the monomorphism $\Obj(\bfX^{\sf duals}) \hookrightarrow \Obj(\bfX)$ is an equivalence.
Using this, Observation~\ref{t15} supplies a map
\[
\Obj(\bfX)
\underset{\rm Obs~\ref{t15}}{~\simeq~}
\Hom_{\fCat_2}\bigl( \Adj , \fB \bfX \bigr)
~.
\]
This map is adjoint to a morphism between category-objects internal to $(\infty,1)$-categories:
\begin{equation}\label{e9}
\Obj(\bfX)
\times
\Adj
\longrightarrow
\fB \bfX
~.
\end{equation}
We now explain the following sequence of canonically $\TT$-equivariant functors between $(\infty,1)$-categories:
\begin{eqnarray}
\nonumber
\Obj(\bfX)^{\pitchfork \SS^1}
\times
c_1
&
\xra{~ \id \times \lag -\infty \xra{!} + \infty \rag~}
&
\Obj(\bfX)^{\pitchfork \SS^1}
\times
\para^{\tl\!\tr}
\\
\nonumber
&
\underset{\id \times {\rm Thm}~\ref{t10}}{~\simeq~}
&
\Obj(\bfX)^{\pitchfork \SS^1}
\times
\displaystyle \int_{\SS^1}
\Adj
\\
\nonumber
&
\underset{\rm Prop~\ref{prop.mappingspace}}{\xla{~\simeq~}}
&
\displaystyle \int_{\SS^1}
\Obj(\bfX)
\times
\displaystyle \int_{\SS^1}
\Adj
\\
\nonumber
&
\underset{\rm Lem~\ref{t4}}{\xla{~\simeq~}}
&
\displaystyle \int_{\SS^1}
\Obj(\bfX)
\times
\Adj
\\
\nonumber
&
\xra{~\displaystyle \int_{\SS^1}(\ref{e9})~}
&
\displaystyle \int_{\SS^1} \fB \bfX
\\
\nonumber
&
\underset{\rm Prop~\ref{t23}}{~\simeq~}
&
\displaystyle \int^\alpha_{\SS^1} \bfX
\\
\label{e10}
&
\underset{(\ref{e6})}{\xra{~\underset{\bfX}\bigotimes~}}
&
\bfX
~.
\end{eqnarray}
The first morphism is the product of the identity map on $\Obj(\bfX)^{\pitchfork \SS^1}$, which is $\TT$-equivariant, and the unique 1-cell between cone-points in $\para^{\tl\!\tr}$, which is $\TT$-invariant.
The second morphism is Theorem~\ref{t10}.
The third identification is Proposition~\ref{prop.mappingspace}, which identifies factorization homology of an $\infty$-groupoid as a mapping space. 
The fourth leftward equivalence is Lemma~\ref{t4}, which states that factorization homology carries products of category-objects to products.  
The fifth functor is factorization homology over $\SS^1$ applied to~(\ref{e9}).
The penultimate functor is the identification of Proposition~\ref{t23} between the $\beta$-version factorization homology and the $\alpha$-version of factorization homology. 
The last functor is the functor~(\ref{e6}), which is $\TT$-invariant.

The composite $\TT$-invariant functor~(\ref{e10}) is, in turn, adjoint to a $\TT$-invariant map
\begin{equation}
\label{e11}
\Obj(\bfX)^{\pitchfork \SS^1}
\longrightarrow
\Hom_{\Cat_1}(
c_1
,
\bfX
)
~.
\end{equation}
By inspection, this $\TT$-invariant map canonically fits into a commutative square:
\[
\xymatrix{
\Obj(\bfX)^{\pitchfork \SS^1}
\ar[rr]^-{(\ref{e11})}
\ar[d]
&&
\Hom_{\Cat_1}(c_1
,
\bfX)
\ar[d]^-{\rm restriction}
\\
\ast
\ar[rr]^-{\lag (\uno , \uno ) \rag}
&&
\Hom_{\Cat_1}(
\partial c_1
,
\bfX
)
}
\]
By definition of the space $\End_{\bfX}(\uno)$, such a commutative diagram is precisely a $\TT$-invariant map
\begin{equation}\label{e12}
L\bfX~:=~\Obj(\bfX)^{\pitchfork \SS^1}
\longrightarrow
\End_{\bfX}(\uno)
~.
\end{equation}
This map~(\ref{e12}) is evidently functorial in $\bfX\in \CAlg(\Cat_1)^{\sf rigid}$, thereby supplying the sought $\TT$-invariant natural transformation:
\[
L
\longrightarrow
\End(\uno)
~.
\]
Finally, Theorem~\ref{r9}(3) gives that the value on the constant loop at $V$ is $\uno\xra{\epsilon\circ \eta}\uno$.
\end{proof}

\appendix

\section{Recollections: the paracyclic category and adjunctions}\label{sec.adjunction}

\subsection{The paracyclic category}

We review basic features of paracyclic category $\para$, introduced by Getzler and Jones in \cite{getzler.jones}. See also Definition~4.2.1 of~\cite{lurie.rotation}.

\begin{definition}
An object in the \bit{paracyclic category} $\para$ is a linearly ordered set $I$ with finite intervals (i.e., for each $i,j\in I$, the set $[i,j] = \{ k\in I \mid i \leq k \leq j\}$ is finite) equipped with an action of the additive group of integers $\ZZ\lacts I$ for which $i< 1 \cdot i$.  
A morphism in $\para$ is a $\ZZ$-equivariant non-decreasing map.   
Composition is composition of maps.
Identities are identity maps.
\end{definition}

\begin{observation}
The $\ZZ$-actions of each object in $\para$ assemble as a $\sB\ZZ$-action on the category $\para$.
Such an action is, for each $\lambda,\lambda'\in \para$, a left $\ZZ$-action on $\Hom_{\para}(\lambda,\lambda')$, 
such that, for each $\lambda,\lambda',\lambda''\in \para$, the composition map
\[
\Hom_{\para}(\lambda',\lambda'')
\times
\Hom_{\para}(\lambda,\lambda')
\xra{~\circ~}
\Hom_{\para}(\lambda,\lambda'')
\]
is equivariant with respect to the homomorphism $\ZZ\times \ZZ\xra{+} \ZZ$.
And, indeed, for $\lambda,\lambda'\in \para$, take the $\ZZ$-action
\[
\ZZ \racts \Hom_{\para}(\lambda,\lambda')
~,\qquad\text{ given by }\qquad
r \cdot f 
\colon 
\lambda \ni 
\ell
\mapsto r\cdot f(\ell)
\in
\lambda'
~.
\]
The requisite equivariance with respect to the homomorphism $\ZZ\times \ZZ\xra{+}\ZZ$ follows from, for each composable pair $\lambda\xra{f} \lambda'\xra{g}\lambda''$ of morphisms in $\para$, and for $r,s\in \ZZ$, and for each $\ell\in \lambda$, the identity in $\lambda''$:
\[
(r\cdot g) \circ (s\cdot f) (\ell)
~=~
(r \cdot g) \bigl( s \cdot f ( \ell) \bigr)
~=~
s \cdot (r \cdot g) \bigl( f(\ell) \bigr)
~=~
(s+r) \cdot g \bigl( f (\ell) \bigr)
~=~
(s+r) \cdot (g \circ f) (\ell)
~.
\]

\end{observation}

For the next observation, we reference the category ${\sf Lin}$ of linearly ordered sets and order-preserving maps between them, as well as the natural lift of the presheaf represented by $[1]\in {\sf Lin}$:
\[
\xymatrix{
{\sf Lin}^{\op}
\ar@{-->}[rr]^-{\un{\Hom}_{\sf Lin}\bigl(-,[1]\bigr)} 
\ar[dr]_-{\Hom_{\sf Lin}\bigl(-,[1]\bigr)} 
&&
{\sf Lin} \ar[dl]^-{\rm forget}
\\
&
\Sets
&
~.
}
\]
\begin{observation}
\label{t34}
There is a natural lift of the composite horizontal functor
\[
\xymatrix{
\para^{\op}
\ar@{-->}[rrr]_-{\simeq}^-{\Hom_{\sf Lin}\bigl(-,[1]\bigr)}
\ar[d]_-{\rm forget}
&&& 
\para   \ar[d]^-{\rm forget } 
\\
{\sf Lin}^{\op}  
\ar[rrr]^-{\un{\Hom}_{\sf Lin}(-,[1])}
&&&
{\sf Lin}
~,
}
\]
which is an equivalence between categories.

\end{observation}

\begin{remark}\label{rem.PD}

Through the identification $\bcD(\SS^1) \xra{\simeq}\para$ of Corollary~\ref{t21}, the identification of Observation~\ref{t34} exchanges, for each non-trivial stratification of $\SS^1$, its set of 0-dimensional strata with its set of 1-dimensional strata, each with its inherited cyclic ordering.
In this sense, the equivalence of Observation~\ref{t34} is a combinatorial implementation of \bit{Poincar\'e duality} for the circle.  
\\

\end{remark}

Consider the functor
\begin{equation}\label{e89}
\bDelta
\longrightarrow
\para
~,\qquad
[p]\mapsto [p]^{\bigstar \ZZ}
~,
\end{equation}
whose value on a finite linearly ordered set is its $\ZZ$-fold join, as it is endowed its resulting order-preserving $\ZZ$-action.
\begin{prop}[Proposition~4.2.8 of~\cite{lurie.rotation}]\label{t74}
The functor $\bDelta \xra{(\ref{e89})} \para$ is initial.

\end{prop}

As an $(\oo,1)$-category, $\bDelta^{\op}$ is sifted.  
Because the property of being sifted is inherited along final functors, 
Proposition~\ref{t74} and Observation~\ref{t34} imply the following.
\begin{cor}\label{s2}
As $(\oo,1)$-categories, both $\para$ and $\para^{\op}$ are both sifted and cosifted.
In particular, their classifying space are contractible: $\sB \para\simeq \ast \simeq \sB(\para^{\op})$. 

\end{cor}

\begin{observation}\label{t3}
The composite functor
\[
\bDelta^{\op} \xra{~(\ref{e89})~} \para^{\op}
\underset{\rm Obs~\ref{t34}}{~\simeq~}
\para
\hookrightarrow 
\para^{\tl}
~,
\]
is final: finality of~(\ref{e89}) is Proposition~\ref{t74}; 
finality of $\para\hookrightarrow \para^{\tl}$ follows using Quillen's Theorem~A from contractibility of $\sB \para$, which is Corollary~\ref{s2}.

\end{observation}

The following technical result is used at the end of the proof of Theorem~\ref{t10}.
A proof of this result is a consequence of some simple observations, which reference the main technical lemma in~\cite{fibrations}.  

\begin{lemma}
\label{t53}
The canonical functor between $(\oo,1)$-categories,
\[
\para^{\tl} \underset{\para} \amalg \para^{\tr}
\xra{~\simeq~}
\para^{\tl \tr}
~,
\]
is an equivalence.

\end{lemma}

\begin{observation}
\label{t52}
Let $\cD$ be an $(\oo,1)$-category.
Lemma~1.11 of~\cite{fibrations} implies the following statements, the first which implies the second.
\begin{enumerate}
\item
The canonical functor between $(\oo,1)$-categories,
\[
\cD^{\tl} 
\underset{\cD}
\amalg
\cD^{\tr}
\longrightarrow
\ast^{\tl} \underset{\ast} \amalg \ast^{\tr}
\xra{\simeq}
\ast^{\tl \tr}
~,
\]
is an exponentiable fibration.

\item
The canonical map among spaces, involving a coend over $\cD$,
\[
\sB \cD
\simeq
\ast
\underset{\cD}
\ot
\ast
=
\Hom_{\cD^{\tl} 
\underset{\cD}
\amalg
\cD^{\tr}}
( -\infty, \bullet)
\underset{\cD}
\bigotimes
\Hom_{\cD^{\tl} 
\underset{\cD}
\amalg
\cD^{\tr}}
( \bullet  , +\infty)
\xra{~\circ~}
\Hom_{\cD^{\tl} 
\underset{\cD}
\amalg
\cD^{\tr}}
( -\infty, +\infty)
~,
\]
is an equivalence.

\end{enumerate}

\end{observation}

%
%
%
%
%
%
%
%

%

\begin{proof}[Proof of Lemma~\ref{t53}]
Using that the classifying space of $\para$ is contractible (Corollary~\ref{s2}), Observation~\ref{t52} implies the result.  

\end{proof}

\subsection{The walking monad and comonad}
\label{sec.monad}

We review some basic categories of ordered sets and their relations to algebra.

\begin{definition}\label{d5}
\begin{itemize}
\item[~]

\item
The \bit{walking monad} is the monoidal category $\sO$ in which an object is a linearly ordered finite sets, a morphism is an order preserving map, and whose monoidal structure is join of linearly ordered sets.
The \bit{walking comonad} is the monoidal category $\sO^{\op}$.  

\item
$\sO_+$ is the left $\sO$-module category of linearly ordered finite sets with a maximum, whose morphisms are those order preserving maps that preserve maxima, and whose left $\sO$-module structure is join of linearly ordered sets.

\item
$\sO_-$ is the right $\sO$-module category of linearly ordered finite sets with a minimum, whose morphisms are those order preserving maps that preserve minima, and whose right $\sO$-module structure is join of linearly ordered sets.

\item
$\sO_{\pm}$ is the category of linearly ordered finite sets with a distinct minimum and maximum, whose morphisms are those order preserving maps that preserve extrema.  

\end{itemize}

\end{definition}

%

\begin{observation}\label{t32}
The canonical inclusion of the non-empty finite linearly ordered sets extends as an equivalence between $(\infty,1)$-categories:
\[
\bDelta^{\tl}
~\simeq~
\sO
\qquad\text{ and }\qquad
(\bDelta^{\op})^{\tr}
~\simeq~
\sO^{\op}
~.
\]
Reversing linear orders defines an equivalence between $\sO$-module categories,
\[
\sO_+ ~\simeq~ \sO_-^{\op}
~.
\]

\end{observation}

\begin{observation}\label{t12}
The representable simplicial set $\Hom_{\bDelta}\bigl(-,[1]\bigr)\colon \bDelta^{\op} \to {\sf Sets}$ canonically factors through the forgetful functor $\sO_{\pm}\to {\sf Sets}$ as an equivalence between categories:
\[
\bdelta^{\op}
\xra{~ \simeq ~}
\sO_{\pm} 
~,\qquad
[p]\mapsto \Hom_{\bDelta}\bigl( [p] , [1] \bigr)
~.
\]

\end{observation}

\begin{observation}
\label{t99}
For each $I \in \sO$, taking joins defines a functor:
\[
\sO^{\times I}
\longrightarrow
\sO_{/I}
~,\qquad
(J_i)_{i\in I}
\longmapsto 
\left(
\underset{i\in I} \bigstar J_i
\to 
\underset{i\in I} \bigstar \ast
= I
\right)
~.
\]
This functor is an equivalence, with inverse given by $(J\xra{f} I)\mapsto \left(  f^{-1}(i) \right)_{i\in I}$.

\end{observation}

The next result records the sense in which $\sO$ and $\sO^{\op}$ corepresent associative algebras and associative coalgebras, respectively.
\begin{prop}\label{t33}
In the monoidal category $\sO$, the final object $\ast \in \sO$ has a unique associative algebra structure; in the monoidal category $\sO^{\op}$, the initial object $\ast \in \sO^{\op}$ has a unique associative coalgebra structure.
For each monoidal $(\infty,1)$-category $\cY$, the resulting functors
\[
\Fun^{\cE_1}( \sO , \cY )
\xra{~\simeq~}
\Alg(\cY)
\qquad\text{ and }\qquad
\Fun^{\cE_1}( \sO^{\op} , \cY )
\xra{~\simeq~}
{\sf cAlg}(\cY)
~,\qquad
F\mapsto F(\ast)
~,
\]
are equivalences between $(\infty,1)$-categories.

\end{prop}

\begin{proof}
By taking categorical opposites, the first statement implies the second.  So we are reduced to proving the first statement. 

Consider the associative $\infty$-operad ${\sf Assoc}^{\ot} \to \bDelta^{\op}$ associated to the associative operad ${\sf Assoc}$.
The monoidal envelope ${\sf Env}({\sf Assoc})$ of this associative $\infty$-operad is the monoidal category presented as the coCartesian fibration:
\[
{\sf Env}({\sf Assoc})^{\ot}
~:=~
\left(
~
{\sf Assoc}^{\ot} \underset{\bDelta^{\op}} \times  \Ar^{\sf act}( \bDelta^{\op}) 
\xra{~\ev_t~}
\bDelta^{\op}
~
\right)
~,
\]
where $\Ar^{\sf act}(\bDelta^{\op}) \subset \Ar(\bDelta^{\op})$ is the full subcategory consisting of the active morphisms, and the fiber product is with respect to the functor $\Ar^{\sf act}(\bDelta^{\op}) \xra{\ev_s} \bDelta^{\op}$ given by evaluation at the source. 
By definition of ${\sf Assoc}^{\ot} \to \bDelta^{\op}$, for $[p]\in \bDelta$, there is a canonical identification fiber $\ev_t^{-1}([p]^\circ) \simeq \sO^{\times p}$ with the $p$-fold product of the category $\sO$.
Furthermore, with respect to such identifications, the coCartesian monodromoy functor over an active morphism in $\bDelta^{\op}$ is given by taking joins of finite linearly ordered sets.  
In particular, the underlying category of ${\sf Env}({\sf Assoc})$ is $\sO$, and its monoidal structure is given by join.
In other words, there is an equivalence between monoidal categories:
\[
{\sf Env}({\sf Assoc})
~\simeq~
\sO
~.
\]
The result then follows from the universal property of monoidal envelopes.

\end{proof}

\subsection{The walking adjunction}

\begin{definition}\label{d6}
The \bit{walking adjunction} is the $(\infty,2)$-category $\Adj$, characterized as receiving a morphism 
\[
c_1 \xra{~\bigl\lag -\xra{\sL}+ \bigr\rag~} \Adj
~,
\] 
with the following universal property:
\begin{enumerate}
\item 
$\sL\in \Mor(\Adj)$ admits a right adjoint;

\item
it is initial with respect to the above property~(1).

\end{enumerate}
In other words, for each $(\infty,2)$-category $\cC$, the evaluation map
\[
\ev_\sL\colon \Hom_{\Cat_2} \bigl( \Adj , \cC \bigr)
\xra{~\simeq~}
\Mor(\cC)^{\sf l.adj}
~,\qquad
F
\mapsto 
F(\sL)
~,
\]
is an equivalence to the subspace of those 1-morphisms in $\cC$ that are left adjoints.

\end{definition}

\begin{notation}
\begin{itemize}
\item[~]

\item
The morphism $c_1 \xra{\lag + \xra{\sR}-\rag} \Adj$ selects the right adjoint to $\sL$~.

\item
The morphism $c_2 \xra{\eta} \Adj$ selects the unit $\id_- \xra{\eta} \sR \circ \sL$~.

\item
The morphism $c_2 \xra{\epsilon} \Adj$ selects the counit $\sL \circ \sR \xra{\epsilon} \id_+$~.

\end{itemize}

\end{notation}

It is not obvious that such an $(\infty,2)$-category exists.
The work~\cite{riehl-verity}, after \cite{schanuel.street}, verifies the existence of $\Adj$, and they identify this $(\infty,2)$-category explicitly as a $(2,2)$-category.

\begin{theorem}[\cite{riehl-verity}, \cite{schanuel.street}]\label{t7}
The $(\infty,2)$-category $\Adj$ admits the following explicit description.
\begin{itemize}
\item
Its space of objects is canonically identified as
\[
\Obj(\Adj)
~\simeq~
\{-,+\}
~,
\]
a 0-type with two path-components.

\item
Its monoidal categories of endomorphisms of each of its two objects are canonically identified as the walking monad and the walking comonad:
\[
\un{\End}_{\Adj}(-)
\underset{\simeq}{\xla{~(\sR\circ \sL)^{\circ I} ~\mapsfrom~ I~}}
\sO
\qquad
\text{ and }
\qquad
\un{\End}_{\Adj}(+)
\underset{\simeq}{\xla{~(\sL\circ \sR)^{\circ I} ~\mapsfrom~ I~}}
\sO^{\op}
~.
\]

\item
Its $\bigl(\un{\End}_{\Adj}(-), \un{\End}_{\Adj}(+)\bigr)$-bimodule category of morphisms from $-$ to $+$, 
and its $\bigl(\un{\End}_{\Adj}(-), \un{\End}_{\Adj}(+)\bigr)$-bimodule category of morphisms from $+$ to $-$, are respectively
\[
\un{\Hom}_{\Adj}(-,+)
\underset{\simeq}{\xla{~\sL \circ (\sR \circ \sL)^{\circ J} ~\mapsfrom~ J^{\tl}~}}
\sO_-~ \simeq ~\sO_+^{\op}
\qquad
\text{ and }
\qquad
\un{\Hom}_{\Adj}(+,-)
\underset{\simeq}{\xla{~(\sR \circ \sL)^{\circ J} \circ \sR ~\mapsfrom~ J^{\tr}~}}
\sO_+ ~\simeq ~\sO_-^{\op}
~.
\]

\end{itemize}

\end{theorem}

\begin{remark}
Biasing the identification $\un{\Hom}_{\Adj}(+,-)\simeq\sO_+$, then the composition functor
\[
\un{\End}_{\Adj}(-)
\times
\un{\Hom}_{\Adj}(+,-)
\xra{~\circ~} 
\un{\Hom}_{\Adj}(+,-)
\]
is equivalent to the left action $\sO\times\sO_+\ra \sO_+$. 
The composition functor
\[
 \un{\Hom}_{\Adj}(+,-)\times\un{\End}_{\Adj}(+) \longrightarrow \un{\Hom}_{\Adj}(+,-)
\]
factors through the right action
\[
\sO_+\times\sO^{\op} 
~\simeq~
\sO_-^{\op}\times\sO^{\op} \longrightarrow 
\sO_-^{\op} ~\simeq~ 
\sO_+
\]
combined with the equivalence $\sO_+\simeq \sO_-^{\op}$ of Observation~\ref{t32}.
\end{remark}

\begin{observation}\label{s6}
In particular, there are fully-faithful functors between $(\infty,2)$-categories,
\[
\fB \sO
~\hookrightarrow~
\Adj
~\hookleftarrow~
\fB\sO^{\op}
~,
\]
and the resulting functor from the coproduct
\[
\fB \sO
\amalg
\fB \sO^{\op}
\longrightarrow
\Adj
\] 
is a monomorphism that is surjective on spaces of objects.  

\end{observation}

Theorem~\ref{t7} lends to a canonical monad and comonad associated to an adjunction, which we record as the following.
\begin{cor}\label{t18}
Let $\Adj \xra{\lag c_- \xra{\sL} c_+\rag } \cC$ select a left adjoint 1-morphism in an $(\infty,2)$-category.  
The composite functor
\begin{equation}\label{e30}
\fB \sO
=
\fB \un{\End}_{\Adj}(-)
\underset{\rm Thm~\ref{t7}}{~\hookrightarrow~}
\Adj
\longrightarrow
\cC
\end{equation}
selects an associative algebra in the monoidal $(\infty,1)$-category $\End_\cC(c_-)$;
the composite functor
\begin{equation}\label{e31}
\fB \sO^{\op}
=
\fB \un{\End}_{\Adj}(+)
\underset{\rm Thm~\ref{t7}}{~\hookrightarrow~}
\Adj
\longrightarrow
\cC
\end{equation}
selects an associative coalgebra in the monoidal $(\infty,1)$-category $\End_\cC(c_+)$.
\\

\end{cor}

\subsection{The walking dual}\label{sec.dual}

\begin{definition}\label{d17}
The \bit{walking dual} is the monoidal $(\infty,1)$-category $\Dual$, equipped with an object $\sL \in \Dual$, with the following universal property:
\begin{enumerate}
\item 
$\sL\in \Dual$ admits a right dual;

\item
it is initial with respect to the above property~(1).

\end{enumerate}
In other words, for each monoidal $(\infty,1)$-category $\cV$, evaluation map,
\[
\ev_\sL\colon \Hom_{\Alg(\Cat_1)} \bigl( \Dual , \cV \bigr)
\xra{~\simeq~}
\Obj(\cV)^{\sf l.dual}
~,\qquad
F
\mapsto 
F(\sL)
~,
\]
is an equivalence to the subspace of those objects in $\cV$ that are left duals.  

\end{definition}

\begin{remark}
Here are two ways of establishing the existence of $\Dual$.
One way is to observe that the functor
\[
\Obj\bigl( -\bigr)^{\sf l.dual}\colon
\Alg(\Cat_1)
\longrightarrow
\Spaces
\]
preserves limits and sufficiently filtered colimits.
Presentability of $\Alg(\Cat_1)$ thereafter ensures the existence of a corepresenting object.  
Alternatively, 
\[
\Dual
~\simeq~
\un{\End}_{\ast \underset{\Obj(\Adj)}\amalg \Adj }(\ast)
\]
is the monoidal $(\infty,1)$-category of endomorphisms in the pushout
\[
\ast \underset{\Obj(\Adj)}\amalg \Adj
\]
of the unique object $\ast \simeq \Obj\Bigl( \ast \underset{\Obj(\Adj)}\amalg \Adj \Bigr)$.

\end{remark}

Consider the canonical morphism between category-objects internal to $(\infty,1)$-categories,
\begin{equation}\label{e15}
\Adj
\longrightarrow
\fB \Dual
~,
\end{equation}
sending the left adjoint $\sL\in \Mor(\Adj)$ to the left dual $\sL\in \Obj(\Dual)$.  
This morphism has the following universal property.  
\begin{observation}\label{t15}
Let $\cV$ be a monoidal $(\infty,1)$-category.
Each solid diagram among category-objects internal to $(\infty,1)$-categories,
\[
\xymatrix{
\Adj \ar[rrrr]^-{ \lag V  \text{ (with right dual)}  \rag } \ar[d]_-{(\ref{e15})}
&&
&&
\fB \cV
\\
\fB \Dual \ar@{-->}[urrrr]_-{\lag V \rag}
&&
&&
~,
}
\]
admits a unique filler.
In other words, restriction along~(\ref{e15}) defines an equivalence between spaces of morphisms
\[
\Obj(\cV)^{\sf l.dual}
\xla{~\simeq~}
\Hom_{\Alg(\Cat_1)}\bigl( \Dual , \cV \bigr)
\xra{~\simeq~}
\Hom_{\fCat_2}\bigl( \Adj , \fB \cV \bigr)
~.
\\
\]

\end{observation}

\begin{notation}\label{d7}
Let $\Dual \xra{\lag V \rag} \cV$ select an object $V\in \cV$ in a monoidal $(\infty,1)$-category that admits a right dual, $V^\vee$.  
We denote the associative algebra~(\ref{e31}) in $\cV$ and the associative coalgera~(\ref{e31}) in $\cV$, respectively, as their underlying objects in $\cV$:
\[
\un{\End}^{\sf alg}_{\cV}(V)
~\simeq~
V^\vee \ot V
\qquad
\text{ and }
\qquad
\un{\End}^{\sf coalg}_{\cV}(V)
~\simeq~
V \ot V^\vee
~.
\]
Should the monoidal structure on $\cV$ be restricted from a symmetric monoidal structure, then we simply denote both this associative algebra and this associative coalgebra as
\[
V^\vee \ot V
~\simeq~
\un{\End}_{\cV}(V)
~\simeq~
V \ot V^\vee
~.
\]

\end{notation}

\begin{remark}
Let $\bfX$ be a symmetric monoidal $(\infty,1)$-category.
The simultaneous algebra and coalgebra structure on $\un{\End}_{\bfX}(V)\in \bfX$ are compatible in the sense that $\un{\End}_{\bfX}(V)$ has the structure of a Frobenius algebra.   
This is structure is expected because $\un{\End}_{\bfX}(V)\in \cV$ is the value of a 1-dimensional 1-framed TQFT on the standardly 1-framed 0-sphere, $\SS^0$.  

\end{remark}

\begin{observation}\label{t19}
Let $\cV$ be a monoidal $(\infty,1)$-category.
The right-lax monoidal functor
$
\cV
\xra{\Hom_{\cV}(\uno , - )}
\Spaces
$
determines a functor between $(\infty,1)$-categories of associative algebras
$
\Alg(\cV)
\to
\Alg(\Spaces)
$
which carries the associative algebra $\un{\End}^{\sf alg}_{\cV}(V)$ to $\End_{\cV}(V)$.
Dually, the right-lax monoidal functor
$
\cV^{\op}
\xra{\Hom_{\cV}(-,\uno)}
\Spaces
$
determines a functor between $(\infty,1)$-categories
$
{\sf cAlg}(\cV)
\to
\Alg(\Spaces)
$
which carries the associative coalgebra $\un{\End}^{\sf coalg}_{\cV}(V)$ to $\End_{\cV}(V^\vee)$.
\\

\end{observation}

\section{Factorization homology}\label{sec.recall.fact}
We first recall some key definitions and constructions concerning factorization homology, as developed in~\cite{fact1} and~\cite{circle}.  
These recollections are not intended to be comprehensive, and we refer a reader to those sources for definitions and in-depth discussion.
We then supply minor technical results that are used in the body of the article above.

\begin{convention}\label{d20}
In this section, $\cV$ is an $(\oo,1)$-category that admits finite limits and geometric realizations.
\end{convention}

\subsection{Recollections of factorization homology}\label{sec.fact.beta}

The $(\oo,1)$-category of \bit{category-objects internal to $\cV$} is the full $\infty$-subcategory
\[
\fCat_1[\cV]
~\subset~
\Fun(\bDelta^{\op} , \cV)
~
\]
consisting of those functors $\bDelta^{\op}\xra{\cC} \cV$ that satisfy the \bit{Segal conditions}, which is to say that the restriction along the monomorphism $\bDelta^{\sf inrt}\hookrightarrow \bDelta$ from the subcategory of \bit{inert} morphisms,
\[
(\bDelta^{\sf inrt})^{\op}
~\hookrightarrow~
\bDelta^{\op}
\xra{~\cC~}
\cV
~,
\]
carries (the opposites of) colimit diagrams to limit diagrams.\footnote{This notation is carried over from~\cite{flagged}, where it is proved that the $(\oo,1)$-category $\fCat_1[\Spaces]$ consists of \bit{flagged $(\infty,1)$-categories}.}

The $(\oo,1)$-category $\bcM:=\cMfd_1^{\sfr}$ is that of compact solidly 1-framed stratified spaces.
Its full $\infty$-subcategory 
\[
\delta\colon
\bcD
~\subset~
\bcM
\]
consists of those objects that are \bit{disk-stratified}: each connected stratum is isomorphic to a Euclidean space.
Refinement maps between underlying stratified spaces supply a class of morphisms in $\bcM$, which are referred to as \emph{refinements}.  
There is a canonical fully-faithful functor
\[
\rho
\colon 
\bDelta^{\op}
\longrightarrow
\bcD
\]
whose image consists of those objects that are refinements of the solidly 1-framed closed disks $\DD^1$ or $\DD^0$.  
\bit{Factorization homology} is right Kan extension along $\rho$ followed by left Kan extension along $\delta$:
\begin{equation}\label{f5}
\displaystyle \int
\colon
\fCat_1[\cV]
\xra{~\rho_\ast~}
\Fun(\bcD , \cV)
\xra{~\delta_!~}
\Fun(\bcM , \cV)
~,\qquad
\cC
\mapsto
\Bigl(
M
\mapsto 
\int_M \cC
\Bigr)
~.
\end{equation}
The following result summarizes some key structural features of factorization homology (see~\cite{circle} for an explanation of the bolded terms).
\begin{theorem}[\cite{fact1}, \cite{circle}]
\label{t4}
Suppose $\cV$ admits finite limits and sifted colimits, and that sifted colimits distribute over products.
Then the Kan extensions~(\ref{f5}) exist.  
Furthermore, for each $M\in \bcM$, the functor $\int_M \colon \fCat_1[\cV] \to \cV$ preserves products.
Lastly, for each $\cC\in \fCat_1[\cV]$, the functor $\int \cC \colon \bcM \to \cV$ carries \bit{closed covers} to limit diagrams, and satisfies \bit{excision}.
\end{theorem}
Via the standard formula for left Kan extensions, for $\cC\in \fCat_1[\cV]$ and $M\in \bcM$, the factorization homology can be identified as a colimit involving the right Kan extension $\rho_\ast \cC$:
\[
\displaystyle \int_M \cC
~\simeq~
\colim\Bigl(
\bcD_{/M}
\xra{~\rm forget~}
\bcD
\xra{~\rho_\ast \cC~}
\cV
\Bigr)
~.
\]
We now explain that, in the case that $\cV = \Cat_n$ for some $n>0$, 
the right Kan extension along $\rho$ is given in terms of the restricted Yoneda functor along $\rho$.
This restricted Yoneda functor factors through the fully-faithful embedding of the $(\oo,1)$-category of $(\infty,1)$-categories as complete Segal spaces, and we denote it as
\[
\fC
\colon 
\bcD^{\op}
\longrightarrow
\Cat_1
~\subset~
\PShv(\bDelta)
~,\qquad
D
\mapsto 
\Bigl(
[p]
\mapsto
\Hom_{\bcD}\bigl( D , \rho([p]) \bigr)
\Bigr)
~.
\]
This functor $\fC$ is fully-faithful, with image consisting of those $(\infty,1)$-categories that are \bit{free on a finite directed graph}.  
In light of this characterization of the image of $\fC$, for $D\in \bcD$ and for $\cC$ an $(\infty,1)$-category, a functor $\fC(D) \xra{\ell} \cC$ is precisely the data of
\begin{itemize}
\item
a map $\sk_0(D) \xra{\ell_0} \Obj(\cC)$ from the $0$-type of 0-dimensional strata of $D$ to the moduli space of objects of $\cC$,

\item
a map $D \smallsetminus \sk_0(D) \xra{\ell_1} \Mor(\cC)$ from the $0$-type of $1$-dimensional strata of $D$ to the moduli space of morphisms of $\cC$,

\item
a commutative diagram among spaces:
\[
\xymatrix{
D \smallsetminus \sk_0(D)
\ar[d]^-{(\ev_s , \ev_t)}
\ar[rr]^-{\ell_1}
&&
\Mor(\cC)
\ar[d]^-{(\ev_s,\ev_t)}
\\
\sk_0(D)
\times
\sk_0(D)
\ar[rr]^-{\ell_0\times \ell_0}
&&
\Obj(\cC)
\times
\Obj(\cC)
.
}
\]

\end{itemize}
Consequently, in the case that $\cV = \Cat_{n-1}$ for some $n> 0$, then, for $\cC\in \Cat_n \subset \fCat_1[\cV]$ an $(\infty,n)$-category, regarded as a category-object internal to $\cV$, the right Kan extension $\rho_\ast \cC$ evaluates as
\begin{equation}\label{f8}
\rho_\ast \cC
\colon
\bcD
\longrightarrow
\Cat_{n-1}
~,\qquad
D
\mapsto
\Hom_{\Cat_n}\bigl(
\fC(D) \wr \bullet
,
\cC
\bigr)
~,
\end{equation}
in where the values are written as presheaves on $\Cat_{n-1}$ which are representable, and for $T\in \Cat_{n-1}$ the $(\infty,n)$-category $\fC(D) \wr T$ that with maximal $(\infty,1)$-subcategory $\fC(D)$, for each pair $x,y\in \Obj\bigl( \fC(D)\bigr)$ of objects, the $(\infty,n-1)$-category of 1-morphisms from $x$ to $y$ is $\Mor\bigl( \fC(D) \bigr) \times T$.

\begin{remark}\label{r5}
The functor~(\ref{f5}) is the $\beta$-version of factorization homology, the second and enhanced version: it pairs compact solidly 1-framed stratified spaces $M$ and category-objects.
To distinguish, we denote the first version of factorization homology (see~\cite{oldfact}) with a superscript ``$\alpha$'': it pairs framed 1-manifolds $M$ and associative algebras.  
Proposition~\ref{t23} of the next subsection relates these two versions.

\end{remark}

Consider the full $\infty$-subcategory 
\begin{equation}\label{f10}
\bcD(M)
~\subset~
\bcD_{/M}
~\subset~
\bcM_{/M}
\end{equation}
consisting of those morphisms $D\xra{r} M$ in which $r$ is a refinement and $D\in \bcD$.
\begin{prop}[\cite{circle}]\label{t48}
Let $M\in \bcM$ be a compact solidly 1-framed stratified space.
The $(\oo,1)$-category $\bcD(M)$ is sifted, and in particular its classifying space $\sB\bcD(M)\simeq \ast$ is contractible.
The $(\oo,1)$-category $\bcD(M)$ is an ordinary category: the space of morphisms in $\bcD(M)$ between any two objects is a 0-type. 

\end{prop}

This $(\oo,1)$-category $\bcD(M)$ plays an important role in factorization homology, as this result indicates.
\begin{lemma}[\cite{circle}]\label{t24}
For each compact solidly 1-framed stratified spaces $M$, 
the fully-faithful functor $\bcD(M) \hookrightarrow \bcD_{/M}$ is final.
In particular, for each category-object $\cC$ internal to $\cV$, 
the canonical morphism in $\cV$,
\[
\colim\Bigl(
\bcD(M)
~\hookrightarrow~
\bcD_{/M}
\xra{~\rm forget~}
\bcD
\xra{~\rho_\ast \cC~}
\cV
\Bigr)
\xra{~\simeq~}
\displaystyle \int_M \cC
~,
\]
is an equivalence.

\end{lemma}

After the above synopsis of factorization homology, Lemma~\ref{t24} supports the following.
\begin{observation}\label{t57}
Consider the case $\cV = \Cat_{n-1}$.
Let $\cC \in \Cat_n \subset \fCat_1[\Cat_{n-1}]$ be an $(\infty,n)$-category, regarded as a category-object internal to $\Cat_{n-1}$.
Let $M\in \bcM$ be a compact solidly 1-framed stratified space.
The composite functor in Lemma~\ref{t24} is identical to the composite functor
\begin{equation}\label{e38}
\bcD(M)
~\hookrightarrow~
\bcD_{/M}
\xra{~\rm forget~}
\bcD
\xra{~\Hom_{\Cat_n}\bigl(
~
\fC(D) \wr \bullet
~,~
\cC
~
\bigr)
~}
\Cat_{n-1}
~.
\end{equation}
An object in $\int_M\cC$ is represented by a disk-refinement $D\to M$ and a functor $\fC(D) \xra{\ell} \cC$.
Such a functor $\ell$ consists of an object $\ell_0(x)\in \Obj(\cC)$ for each connected $0$-dimensional stratum $x\in \sk_0(D)$,  
a 1-morphism $\bigl(\ell_0(x)\xra{\ell_1(U)} \ell_0(y)\bigr) \in \Mor(\cC)$ for each connected 1-dimensional stratum $U\subset D\smallsetminus \sk_0(D)$ whose point-set boundary is $\{x,y\} \subset D$, and where the directionality of the morphism $c_U$ is determined by the 1-framing of $U$.

\end{observation}

\begin{prop}[1-dimensional nonabelian Poincar\'e duality]
\label{prop.mappingspace}
Let $Z\in \Spaces \subset \Cat_1$ be a space, regarded as an $(\infty,1)$-category.
There is a canonical equivalence between functors from $\bcM$ to $\Spaces$,
\[
\int_{M}Z 
~\simeq~
\Map(M,Z)
~.
\]

\end{prop}

\begin{proof}
By definition of $\bcD$, there is a forgetting solid 1-framing defines a functor $\bcD \xra{\rm forget} \cBun$ to the $(\oo,1)$-category introduced in~\cite{striat} that classifies proper constructible bundles.
In turn, there is a functor $\cBun \xra{\sB \Exit} \Spaces^{\op}$ which classifies the fact that, for $X\to K$ a proper constructible bundle between stratified spaces, the fiber-wise classifying space $\sB^{\sf fib} \exit(X) \to \exit(K)$ is a right fibration.
Observe the left commutative diagram among $(\oo,1)$-categories
\begin{equation}
\label{f6}
\xymatrix{
\bcD^{\op} 
\ar[rr]^-{\fC}
\ar[d]_-{\rm forget}
&&
\Cat_1
\ar[d]^-{\sB}
&
&
\bcD(M)^{\op} 
\ar[rr]^-{\fC}
\ar[d]_-{!}
&&
\Cat_1
\ar[d]^-{\sB}
\\
\cBun^{\op}
\ar[rr]^-{\sB \Exit}
&&
\Spaces
,
&
\text{ ; }
&
\ast
\ar[rr]^-{\lag M \rag}
&&
\Spaces
,
}
\end{equation}
which is the functorial identification of $\sB \fC(D) \simeq \sB \exit(D) \simeq D$ of the classifying space of $\fC(D)$ with the underlying space of $D$, for each $D\in \bcD$;
the right commutative diagram follows.

This left commutative diagram~(\ref{f6}) lends to an identification of the right Kan extension $\rho_\ast Z$:
\begin{equation}
\label{f7}
\rho_\ast Z
\underset{(\ref{f8})}{~\simeq~}
\Hom_{\Cat_1}\bigl( \fC(-) , Z \bigr)
~\simeq~
\Hom_{\Spaces}\bigl( \sB \fC(-) , Z \bigr)
\underset{\rm left~(\ref{f6})}{~\simeq~}
\Hom_{\Spaces}\bigl( - , Z \bigr)
~=~
\Map(-,Z)
~,
\end{equation}
in which the middle identification uses that the $(\infty,1)$-category $Z$ is an $\infty$-groupoid.  
Therefore, the right commutative diagram~(\ref{f6}) lends to an identification of the restriction of $\rho_\ast Z$,
\begin{equation}
\label{f11}
\bcD(M)
\xra{~\rm forget~}
\bcD
\xra{\rho_\ast Z}
\Spaces
~,\qquad
(D\xra{\sf ref} M)
\mapsto 
\Map(M,Z)
~,
\end{equation}
as the constant functor at the space $\Map(M,Z)$.
So the factorization homology $\int Z$, which is the left Kan extension $\delta_! \rho_\ast Z$, is identified as
\begin{eqnarray}
\nonumber
\displaystyle \int_M Z
&
\underset{\rm Lem~\ref{t24}}{~\simeq~}
&
\colim\Bigl(
\bcD(M)
\xra{~\rho_\ast Z~}
\Spaces
\Bigr)
\\
\nonumber
&
\underset{(\ref{f7})}{~\simeq~}
&
\colim\Bigl(
\bcD(M)
\xra{~\Map(-,Z)~}
\Spaces
\Bigr)
\\
\nonumber
&
\underset{(\ref{f11})}{~\simeq~}
&
\colim\Bigl(
\bcD(M)
\xra{!}
\ast
\xra{\lag \Map(M,Z) \rag}
\Spaces
\Bigr)
\\
\nonumber
&
~\simeq~
&
\sB \bcD(M) \times  \Map(M,Z)
\\
\nonumber
&
\underset{\rm Prop~\ref{t48}}{\xra{~\simeq~}}
&
\Map(M,Z)
~.
\\
\end{eqnarray}

\end{proof}

\subsection{Comparison between the $\alpha$ and $\beta$ versions of factorization homology}

On their overlapping domains of definition, the $\beta$-version of factorization homology agrees with the $\alpha$-version of factorization homology.  
This is articulated as Proposition~\ref{t23} below.

Recall the symmetric monoidal $(\oo,1)$-category $\Mfld_1^{\fr}$, which is the subject of~\cite{oldfact}.
An object in $\Mfld_1^{\fr}$ is a framed 1-manifolds (without boundary, and finitary).
The space of morphisms between two is the space of framed open embeddings. 
The symmetric monoidal structure is disjoint union.  
Assigning to each compact solidly 1-framed stratified space $M$ its 1-dimensional strata defines a symmetric monoidal functor to $\Mfld_1^{\fr}$ from the monoidal $\infty$-subcategory $\bcM^{1}\subset \bcM$ consisting of all objects yet only those morphisms generated by refinements and creation morphisms that are isomorphisms over 1-dimensional strata:
\begin{equation}\label{e50}
\bcM \supset \bcM^{1}
\longrightarrow
\Mfld_1^{\fr}
~,\qquad
M\mapsto 
M \smallsetminus \sk_0(M)
~.
\end{equation}
In particular, for each $M\in \bcM$, there is a functor~(\ref{e50}) restricts as a functor between 
$\infty$-overcategories:
\begin{equation}\label{f9}
\bcD(M)
\xra{~(\ref{e50})_|~}
\Disk^{\fr}_{1/M\smallsetminus \sk_0(M)}
~,
\end{equation}
which is equivariant with respect the action of $\Aut_{\bcM}(M) \ra \Aut_{\Mfld_1^{\fr}}(M\smallsetminus\sk_0(M))$.
\begin{observation}
\label{t20}
The functor~(\ref{f9}) is fully-faithful if $M$ is 1-diimensional. The essential image consists of those $U\hookrightarrow M\smallsetminus \sk_0(M)$ that are surjective on path-components.

\end{observation}

\begin{cor}\label{t21}
In the case that $M= \SS^1$, the functor~(\ref{f9}) is fully-faithful and witnesses freely adjoining an initial object:
\[
\Bigl(
\bcD(\SS^1)
\hookrightarrow
\bcD(\SS^1)^{\tl}
\Bigr)
~\underset{(\ref{f9})}{\simeq}~
\Bigl(
\bcD(\SS^1)
\hookrightarrow
\Disk^{\fr}_{1/\SS^1}
\Bigr)
~.
\]
In particular, Lemma~\ref{t31} lends to a canonical $\TT\simeq \sB\ZZ$-equivariant equivalence:
\[
\bcD(\SS^1)
\xra{~\simeq~}
\para
~,\qquad
(D\xra{\sf ref} \SS^1)
\mapsto 
{\sf exp}^{-1}\bigl( D\smallsetminus \sk_0(D) \bigr)
~.
\]

\end{cor}

After Corollary~\ref{s2}, Corollary~\ref{t21} immediately gives the following, which is also stated as Proposition~\ref{t48} above.  
\begin{cor}\label{t22}
The $(\oo,1)$-category $\bcD(\SS^1)$ has the following properties.
It is sifted; its classifying space is contractible; it is an ordinary category.

\end{cor}

\begin{remark}
Anna Cepek's PhD thesis~\cite{cepek.Ran.circle} proves a $\TT\simeq \sB\ZZ$-equivariant identification $\exit\bigl( {\sf Ran}^{\sf u}(\SS^1) \bigr) \simeq \para^{\op}$ of the exit-path $(\oo,1)$-category of a unital version of the Ran space of $\SS^1$.
The identification of Corollary~\ref{t21} is related to this identification of Cepek's, and imports the essential idea demonstrated there of considering $\exp^{-1}(-)$.  

\end{remark}

In the next result, we consider the Cartesian symmetric monoidal structure on the $(\oo,1)$-category $\cV$, and we reference the functor
\[
\fB
\colon
\Alg(\cV)
~\hookrightarrow~
\fCat_1[\cV]
~\subset~
\Fun(\bDelta^{\op} , \cV)
~,\qquad
A\mapsto \bBar_\bullet(A)
~.
\]

\begin{prop}\label{t23}
There is a canonical identification between functors from $\Alg(\cV)$ to $\Fun(\bcM^{1} , \cV)$:
\[
\int^\alpha_{-\smallsetminus \sk_0(-)}
(-)
~\simeq~
\int_- \fB (-)
~.
\]
In particular, for each associative algebra $A \in \Alg(\cV)$, there is a canonical $\TT$-equivariant identification
\[
\int^\alpha_{\SS^1} A
~\simeq~
\int_{\SS^1} \fB A
~.
\]

\end{prop}
\begin{proof}
Let $A\in \Alg(\cV)$, and let $M\in \bcM^1$.  
By definition, the two forms of factorization homology evaluate as colimits
{\Small
\[
\int^\alpha_{M\smallsetminus \sk_0(M)} A 
\simeq 
\colim\Bigl(\Disk^{\fr}_{1/M\smallsetminus \sk_0(M)}
\xra{\rm forget}
\Disk_1^{\fr}
\xra{A}\cV\Bigr)
\qquad
\text{ and }
\qquad
\int_{M} \fB A 
\simeq 
\colim\Bigl(\bcD_{/M} \xra{\rm forget} \bcD \xra{\rho_\ast \fB A}\cV\Bigr)
~.
\]
}
Observe the canonically commutative diagram among $(\oo,1)$-categories:
\[
\xymatrix{
\bcD(M)
\ar[rr]
\ar[d]_-{(\ref{f9})}
&&
\bcD_{/M}
\ar[rr]^-{\rm forget}
&&
\bcD 
\ar[d]^-{\rho_\ast \fB A}
\\
\Disk^{\fr}_{1/M\smallsetminus \sk_0(M)}
\ar[rr]^-{\rm forget}
&&
\Disk_1^{\fr}
\ar[rr]^-{A}
&&
\cV
.
}
\]
Lemma~\ref{t24} states that the top left horizontal functor is final.
This finality supplies a canonical morphism in $\cV$:
\[
\int_M \fB A
\longrightarrow
\int^\alpha_{M\smallsetminus\sk_0(M)} A
~.
\]
The functoriality of this morphism in both arguments $A\in \Alg(\cV)$ and $M\in \bcM^1$ is evident.
The result is established upon showing this morphism is an equivalence.
For this, it is enough to show that the functor~(\ref{f9}) is final.

Observation~\ref{t20} identifies~(\ref{f9}) as the inclusion of the full $\infty$-subcategory consisting of those $U\hookrightarrow M \smallsetminus \sk_0(M)$ that are surjective on path-components.
Consider the $\infty$-subcategory $\Disk^{\fr,\sf surj}_{1/M\smallsetminus \sk_0(M)}  \ra  \Disk^{\fr}_{1/M\smallsetminus \sk_0(M)}$ consisting of those $U\hookrightarrow M \smallsetminus \sk_0(M)$ that are surjective on path-components, and those morphisms $U \hookrightarrow V \hookrightarrow M \smallsetminus \sk_0(M)$ in which $U\hookrightarrow V$ is also surjective on path-components.
Note that, via Observation~\ref{t20}, $\bcD(M)$ is the smallest full $\infty$-subcategory of $\Disk^{\fr}_{1/M\smallsetminus \sk_0(M)}$ containing the image of $\Disk^{\fr, \sf surj}_{1/M\smallsetminus \sk_0(M)}$.
By Lemma~2.3.2 of~\cite{pkd}, the monomorphism 
$\Disk^{\fr, \sf surj}_{1/M\smallsetminus \sk_0(M)}
\hookrightarrow
\Disk^{\fr}_{1/M\smallsetminus \sk_0(M)}
$
is final. 
By Quillen's Theorem~A, this implies that if $\cC\ra \Disk^{\fr}_{1/M\smallsetminus \sk_0(M)}$ is a full $\oo$-subcategory which contains the image of the functor $\Disk^{\fr,\sf surj}_{1/M\smallsetminus \sk_0(M)}\ra  \Disk^{\fr}_{1/M\smallsetminus \sk_0(M)}$, then the functor $\cC\ra \Disk^{\fr}_{1/M\smallsetminus \sk_0(M)}$ is again final. 
Setting $\cC = \bcD(M)$ implies the result.
\end{proof}

\end{document}